\documentclass[a4paper,12pt,reqno]{amsart}
\usepackage{vmargin}
\usepackage[pdftex]{graphicx}
\usepackage{amsfonts}
\usepackage{amssymb}
\usepackage{mathrsfs}
\usepackage{amsmath}
\usepackage{amsthm}
\usepackage[shortlabels]{enumitem}
\usepackage{nomencl}
\usepackage[bookmarks=true]{hyperref}   
\usepackage{esint}
\usepackage{dsfont}
\usepackage{upgreek}
\usepackage{xcolor}
\usepackage{slashed}
\usepackage{scalerel}
\usepackage[utf8]{inputenc}

\makeatletter
\renewcommand\section{\@startsection{section}{1}{\z@}%
                       {-3\p@ \@plus -4\p@ \@minus -4\p@}%
                       {3\p@ \@plus 4\p@ \@minus 4\p@}%
                      {\normalfont\normalsize\centering\scshape}}
\makeatother

\numberwithin{equation}{section} 	

\hypersetup{
    colorlinks,%
}

\newcommand{\Rosen}{Ros\'en}	

\author{Lashi Bandara}
\author{Hemanth Saratchandran}

\title[Ess self-adj. of f.o. diff. ops. on low-reg. metrics]{Essential self-adjointness of powers
of first-order differential operators on non-compact manifolds with low-regularity metrics}
\date{\today}
\address{Lashi Bandara, Mathematical Sciences,
Chalmers University of Technology and University of Gothenburg, SE-412 96, Gothenburg, Sweden}
\urladdr{\href{http://www.math.chalmers.se/~lashitha}{http://www.math.chalmers.se/~lashitha}}
\email{\href{mailto:lashi.bandara@chalmers.se}{lashi.bandara@chalmers.se}}

\address{Hemanth Saratchandran, 
Beijing International Centre for Mathematical Research, Peking University, Beijing 100871, PR China}
\email{\href{mailto:hemanth.saratchandran@gmail.com}{hemanth.saratchandran@gmail.com}}


\def\colour{\colour}

\def\colour{\color}

\newtheorem{theorem}{Theorem}[section]
\newtheorem{corollary}[theorem]{Corollary}
\newtheorem{lemma}[theorem]{Lemma}
\newtheorem{proposition}[theorem]{Proposition}

\newtheorem{definition}[theorem]{Definition}
\newtheorem{remark}[theorem]{Remark}

\newtheorem{example}[theorem]{Example}

\newcommand{\mdot}{\cdot}

\newcommand{\cbrac}[1]{\left(#1\right)}
\newcommand{\bbrac}[1]{\left[#1\right]}
\newcommand{\dbrac}[1]{\left\{#1\right\}}
\newcommand{\modulus}[1]{\left|#1\right|}
\newcommand{\set}[1]{\dbrac{#1}}

\newcommand{\dom}{ {\mathcal{D}}}
\newcommand{\ran}{ {\mathcal{R}}}
\newcommand{\nul}{ {\mathcal{N}}}

\newcommand{\comp}{\, \circ\, }

\newcommand{\e}{\mathrm{e}}

\newcommand{\R}{\mathbb{R}}
\newcommand{\C}{\mathbb{C}}

\newcommand{\Na}{\ensuremath{\mathbb{N}}}

\newcommand{\script}[1]{\mathscr{#1}}

\renewcommand{\emptyset}{\varnothing}

\newcommand{\intersect}{\cap}

\newcommand{\rest}[1]{{{\lvert_{}}_{}}_{#1}}
\newcommand{\close}[1]{\overline{#1}}		

\renewcommand{\epsilon}{\varepsilon}

\renewcommand{\phi}{\varphi}

\newcommand{\graph}{\script{G}}		

\newcommand{\tensor}{\otimes}

\newcommand{\comm}[1]{\bbrac{#1}}		

\newcommand{\norm}[1]{\| #1 \|}			

\newcommand{\spt}[1]{{\rm spt} {\text{ }}#1}	
\DeclareMathOperator{\esssup}{esssup}

\DeclareMathOperator{\tr}{tr}			

\DeclareMathOperator{\len}{\ell}			

\DeclareMathOperator{\divv}{div}		

\newcommand{\cut}{\ \llcorner\ }			

\newcommand{\Forms}[1][{}]{\mathbf{\Omega}^{#1}}		
\newcommand{\Tensors}[1][{}]{{\mathcal{T}}^{(#1)}}	
\newcommand{\Sect}{\mathbf{\Gamma}}		
\newcommand{\tanb}{{\rm T}}		
\newcommand{\cotanb}{{\rm T}^\ast}	

\newcommand{\pullb}[1]{#1^\ast}			

\DeclareFontFamily{OT1}{restrictfont}{}
\DeclareFontShape{OT1}{restrictfont}{m}{n}{<-> fmvr8x}{}

\newcommand{\adj}[1]{{#1}^\ast}			
\newcommand{\extd}{{\rm d}}			
\newcommand{\intd}{{\updelta}}

\newcommand{\inprod}[1]{\left\langle #1 \right\rangle}	
\newcommand{\conn}[1][{}]{{\nabla_{{#1}}}}		
\newcommand{\Leb}[1][{}]{\script{L}^{#1}}			

\newcommand{\Cliff}[1][{}]{\Delta^{#1}}		
\DeclareMathOperator{\cliff}{{\scaleobj{0.5}{\triangle}\, }}	
\DeclareMathOperator{\Spin}{Spin}			
\DeclareMathOperator{\Spinors}{\slashed{\Delta}}	
\newcommand{\spin}[1]{\slashed{#1}}		
\newcommand{\Prin}[1]{\mathrm{P}_{#1}}		

\newcommand{\bddlf}{\mathcal{L}} 	

\newcommand{\Lp}[2][{}]{{\rm L}^{#2}_{\rm #1}}		
\newcommand{\Ck}[2][{}]{{\rm C}^{#2}_{\rm #1}}		
\newcommand{\Sob}[2][{}]{{\rm W}^{#2}_{\rm #1}}		
\newcommand{\SobH}[2][{}]{\Sob[#1]{#2,2}}

\newcommand{\convolve}{\, \ast\, }

\newcommand{\iden}{{\mathrm{I}}}

\newcommand{\Hil}{\script{H}}			
\newcommand{\Lap}{\Delta}			

\newcommand{\sA}{\script{B}}

\newcommand{\sC}{\script{C}}

\newcommand{\cV}{\mathcal{V}}
\newcommand{\cM}{\mathcal{M}}

\newcommand{\cS}{\mathcal{S}}

\newcommand{\mg}{\mathrm{g}}
\newcommand{\mgt}{\tilde{\mg}}
\newcommand{\mh}{\mathrm{h}}

\newcommand{\Mul}{{\mathrm{M}}}

\newcommand{\Div}{\mathrm{L}}

\DeclareMathOperator{\Sym}{Sym}

\newcommand{\B}{\mathrm{B}}

\newcommand{\Dir}{{\rm D} }

\newcommand{\met}{\uprho}		

\newcommand{\rep}{\cdot}

\newcommand{\RNum}[1]{\uppercase\expandafter{\romannumeral #1\relax}}

\newcommand{\Sp}{\mathcal{S}}

\begin{document}

\maketitle
\vspace*{-2em}
\begin{abstract}

We consider first-order differential operators 
with locally bounded measurable coefficients on 
vector bundles with measurable coefficient metrics.  
Under a mild set of assumptions,
we demonstrate the equivalence between the essential 
self-adjointness of such operators to a negligible
boundary property. When the operator possesses higher
regularity coefficients, we show that 
higher powers are essentially self-adjoint
if and only if this condition is satisfied.
In the case that the low-regularity
Riemannian metric induces a complete length space, we
demonstrate essential self-adjointness of
the operator and its higher powers up to the regularity 
of its coefficients. We also present applications
to Dirac operators on Dirac bundles
 when the metric is non-smooth.
\end{abstract}
\tableofcontents
\vspace*{-2em}

\parindent0cm
\setlength{\parskip}{\baselineskip}

\section{Introduction}

The problem of determining essential self-adjointness of
smooth coefficient first-order differential operators, 
as well as their powers, 
is an important and well studied topic. 
This paper considers similar problems but in the
context of non-smooth coefficients. More precisely,
we consider  first-order  symmetric
differential operators $\Dir$, as well as their powers, on smooth vector bundles
 $\cV$, over smooth, noncompact manifolds $\cM$. 
We allow for the coefficients of the operator as
well as the metrics on the bundle and the manifold to be non-smooth. 
Our primary focus is to understand the relationship between the regularity
of the coefficients of $\Dir$ and the essential self-adjointness
of powers of $\Dir$.


One of the primary motivations for studying the essential self-adjointness of a differential operator $\Dir$
comes from the fact that it allows one to build a functional calculus (of Borel functions) for the closure of that operator. Such
a functional calculus can then, for instance, be used to 
build a heat operator, $e^{-t\Dir}$ for $t > 0$, and a Schrodinger 
operator, 
$e^{\imath t\Dir}$ for $t \in \R$. These in turn can then be used to solve the heat equation $u_t + \Dir u = 0$, and the 
Schrodinger equation $u_t + \imath \Dir u = 0$ respectively. It is in the construction of such solutions to differential 
equations that
makes essential self-adjointness an indispensable property.   


There is a plethora of historical 
literature surrounding this subject, and therefore, we confine ourselves
to presenting only the relevant references to our work.
From our point of view, it was Gaffney who in \cite{Gaffney} 
made a first significant contribution 
by establishing the essential self-adjointness of the 
Hodge Laplacian $(\extd\intd + \intd\extd)$ under
a so-called \emph{negligible boundary} condition.
 Moreover, Gaffney allows his manifold to 
be $\Ck{k}$ and incomplete.  The next relevant reference to us is the work of 
 Wolf in \cite{Wolf},
where he  studies the essential self-adjointness of general 
Dirac operators and their squares. Moreover, Cordes in \cite{Cordes}
obtains the essential self-adjointness of all powers
of the Laplace-Beltrami operator $\Lap^m$ on functions. 
Later, in \cite{Chernoff}, Chernoff
studies conditions under which essential self-adjointness of \emph{all} 
powers of first-order operators are obtained.
These last three references assume both completeness and
smoothness of their metrics and the coefficients of the operators.


In the last few decades, there has been an interest in the study of smooth manifolds admitting non-smooth
metrics. For example, Anderson and Cheeger (see \cite{AC}) were able to show
that certain smooth manifolds admitting $C^{\alpha}$ metrics can be seen to arise as limits of smooth manifolds
with smooth metrics, satisfying bounds on their Ricci curvature, injectivity radius, and volume.
In \cite{CH}, Chen and Hsu studied gradient estimates for weakly harmonic functions
on smooth complete manifolds with a Lipschitz continuous metric, and were able to extend a result of Yau
to this setting. Non-smooth metrics have also arisen in studies associated to the Ricci flow through
works of Simon \cite{Simon}, and Chen and Ding \cite{CD}. 
The work of these authors motivates the study of differential operators on such non-smooth spaces. 


The study we present here is motivated by the question of whether one can carry out functional calculus
constructions, e.g. building a  Schrödinger   operator $e^{\imath t\Dir}$ for $t > 0$, in the above mentioned
non-smooth settings. As essential self-adjointness is the key property to carrying out such functional calculi
constructions  (of Borel functions), we solely focus on this property. While the work of \cite{Cordes} and \cite{Chernoff}
are perfectly suited for the smooth situation, they are wholly inadequate in non-smooth settings such as
those mentioned in the above references. 
Our work can therefore be seen as an attempt to recover results
similar to that of \cite{Cordes} and \cite{Chernoff}, but without assuming smoothness of the operator and the 
metric. A priori, our differential operator $\Dir:\Ck{0,1}(\cV) \to \Lp[loc]{\infty}(\cV)$,
and such operators can arise, for instance, as operators
built from Levi-Civita connections associated  to  Lipschitz metrics.
Under a set of mild assumptions \ref{D:First}-\ref{D:Last},  
we proceed by demonstrating that the essential self-adjointness
of this operator is equivalent to the 
\emph{negligible boundary} property formulated with 
respect to the operator $\Dir$ (see Definition \ref{Def:NegBdy}).
 This property is trivially satisfied
on compact manifolds and hence, our analysis is exclusively 
carried out in the non-compact setting.  
If the differential operator has $\Ck{m}$-coefficients,
then for $l \leq m +1 $, 
we demonstrate that the essential self-adjointness of $\Dir^l$
is equivalent to this aforementioned property for the operator
$\Dir^l$. Indeed, we do not expect essential self-adjointness, 
in general, to survive for orders $l > m +1$ for a $\Ck{m}$-coefficient operator.
 This is partly due to the fact that, 
even for an arbitrary smooth compactly supported $u$,
 $\Dir^l u$ exists only distributionally.

In \S\ref{S:EssNeg}, we demonstrate this equivalence under the mild set of assumptions 
\ref{D:First}-\ref{D:Last}, primarily motivated by features we would
expect from elliptic operators.  We emphasise that in demonstrating
this equivalence,  we do not make any assumptions about the 
completeness of the underlying space. We
proceed abstractly in order to emphasise this point,
and in fact, carry out this analysis
on spaces with so-called \emph{rough metrics}. These are 
Riemannian-like metrics that have merely measurable coefficients, 
which are comparable against induced Euclidean metrics 
in small enough charts.  A priori, such a metric only 
induces a well defined measure, and it is unknown 
whether there is a naturally  associated notion of length. 
The significance of the negligible
boundary criterion is that it can still be formulated 
on this measure space without alluding to a length structure.

In \S\ref{S:HtoL}, at this same level of generality, 
we are also able to 
show that if $\Dir^l$ is essentially self-adjoint,
then so are its lower powers. This is done via 
functional calculus and operator theoretic methods
to emphasise the fact that  this assertion makes no 
geometric demands.

Geometry begins to play a significant role
when attempting to boost essential self-adjointness
from lower to higher powers. 
We carry out this study in \S\ref{S:Ess},
and the  way in which we do this is to establish the negligible boundary 
condition for higher powers via a bootstrapping procedure. 
Here, we are forced to assume completeness, 
but we are able to allow for a restricted
class of rough metrics that induce a length space. 
This is still a large and significant class of non-trivial
metrics which, for instance, include bi-Lipschitz pullbacks
of  smooth,  complete  metrics. We emphasise that it is the lack of 
regularity that forces us away from  the  wave
technique in \cite{Chernoff} or 
the PDE technique in \cite{Cordes}, as these techniques
seem to require smoothness in a crucial way. 
By generalising
the ideas of \cite{Wolf},  we demonstrate  a certain 
elliptic estimate for lower powers of our operator
from knowing that the higher powers have finite $\Lp{2}$-energy.
This in turn allows us to obtain negligible boundary 
for higher powers via the first power.

In \S\ref{S:App}, we consider applications to 
Dirac and other elliptic differential operators.
As a first consequence, 
we highlight the following result for Spin manifolds with non-smooth 
metrics.
 
\begin{theorem}
\label{Thm:Main2}
Let $\mg$ be a $\Ck{0,1}$ complete metric
on a smooth Spin manifold $\cM$, with a spin
structure $\Prin{\Spin}(\cM)$. Let 
$\Spinors \cM = \Prin{\Spin} \times_\eta \Spinors\R^n$, 
where $\eta:\Spin_n \to \bddlf(\Spinors \R^n)$ is the
usual complex representation. 
Then, the associated \emph{Atiyah Singer Dirac 
operator} $\spin{\Dir}$ is essentially self-adjoint on $\Ck[c]{\infty}(\cV)$.
If the metric is $\Ck{m}$ for $m \geq 1$, then $\spin{\Dir}^l$
is essentially self-adjoint on $\Ck[c]{\infty}(\cV)$ for $l \leq m$.
\end{theorem}

Theorem 
\ref{Thm:Main2} is a particular instance of more
general results we obtain for Dirac operators on general Dirac bundles. 

As a second highlight theorem, we present the following
consequence of our work to general symmetric elliptic operators. 

\begin{theorem}
\label{Thm:Main1}
Let $\cV$ be a smooth bundle over a smooth manifold $\cM$ with 
continuous metrics $\mh$ and $\mg$. Suppose that  $\Dir$ is 
a first-order elliptic  differential operator with $\Ck{m}$-coefficients, $m \geq 0$,
that is symmetric on $\Ck[c]{\infty}$ in $\Lp{2}(\cV)$.
Then, $\Dir^k$ on $\Ck[c]{\infty}(\cV)$ for  $k = 1, \dots, m+1$  is essentially self-adjoint. 
\end{theorem}

Our theorems generalise the consequences of the work of Wolf in \cite{Wolf} 
and Chernoff in \cite{Chernoff} to Dirac operators
to settings where smoothness assumptions on the metrics
are relaxed.  In Theorem \ref{Thm:Main1},
we  recover the results of Chernoff from  \cite{Chernoff}
in the elliptic setting, and we also
dispense with the \emph{local velocity of propagation} condition
assumed in that paper.

\section*{Acknowledgements}
Hemanth Saratchandran wishes to acknowledge support from the Beijing International Centre
for Mathematical Research, and the Jin Guang Mathematical Foundation.
Lashi Bandara was supported by the Knut and Alice Wallenberg foundation, KAW 2013.0322 
postdoctoral program in Mathematics for researchers from outside Sweden.
The authors would like to thank the referee for suggestions
which helped improve the manuscript. 
\section{Setup and main results}

\subsection{Notation}

Throughout this article, we will use the analysts inequality
$a \lesssim b$ to mean that $a \leq C b$ with $C > 0$, as well
as the analysts equivalence $a \simeq b$. The support of a section (or
function) $f$ will be denoted by $\spt{f}$. 
Unless otherwise stated, we will assume Einstein summation convention 
throughout this paper. That is, whenever a raised index appears multiplicatively against
a lowered index, we assume summation over that index. By \emph{precompact set}, 
we mean a set whose closure is compact. 
 
\subsection{Vector bundles over manifolds}

Let $\cM$ be a connected, non-compact, smooth  manifold 
and $\pi_{\cV}: \cV \to \cM$ be a smooth vector bundle
of finite dimension $\dim \cV = N$. By $\cV_x$, we denote the fibre
over the point $x \in \cM$, given by $\pi_{\cV}^{-1}\set{x}$.

The locally Euclidean structure allows us to define
spaces of regularity in a local sense.
The spaces $\Ck{k,\alpha}(\cV)$ will be used to denote $k$
times  differentiable
sections  for which the $k$-th derivative is 
 $\alpha$-H\"older continuous \emph{locally}.
The notation $\Ck[c]{k,\alpha}(\cV)$ will be such sections that
are compactly supported. 

 Moreover,  this locally Euclidean structure 
allows  us to import measure-theoretic notions in the
absence of a metric: we say that a set $A$ is measurable
if $\phi(A \intersect U)$ is   Lebesgue measurable  for all coordinate
charts  $(U, \phi)$  with  $U \intersect A \neq \emptyset$.
This allows us to define measurable functions $f: \cM \to \C$, 
and by $\Sect(\cV)$, we denote measurable sections over $\cV$ 
to be sections $v = v_i e^i$ in continuous  local frames $\set{e^i}$ 
with $v_i$ a measurable function. See \cite{BRough} for 
a detailed construction. 

Using our notion of measurability, we define $\Lp[loc]{1}(\cV)$ to be
measurable sections $v$ such that $\psi^{-1} v \in \Lp[loc]{1}(U, \C^N)$
 over local trivialisations $\psi: U \times \C^N \to \cV$.
 The definition of local Sobolev spaces follow similarly:
we say that $u \in \Sob[loc]{k,p}(\cV)$ if  $\psi^{-1} u \in \Sob[loc]{k,p}(U, \cV)$.

\subsection{Rough metrics}

Since our goal is to study \emph{global} differential operators over $\cV$,
on letting $\adj{\cV}$ be the dual bundle of $\cV$,
we define the   following notion of a metric tensor. 

\begin{definition}[Rough metrics and bundle rough metrics]
Let $\mh \in \Sect(\adj{\cV} \tensor \adj{\cV})$ be
real and symmetric, and suppose that for each $x \in \cM$, 
there exists a trivialisation $(U_x,\psi_x)$ containing $x$ 
and a constant $C = C(U_x) \geq 1$ satisfying:
$$ C^{-1} \modulus{u}_{\mh(y)} \leq \modulus{\psi_x(y) u}_{\delta} \leq C \modulus{u}_{\mh(y)}$$
for almost-every $y \in U_x$  and $u \in \cV_x$, 
where $\delta$ denotes the Euclidean metric in $\C^N$. 
Then, we say that such a metric is a \emph{bundle rough metric},
and such a trivialisation is said to satisfy the \emph{local comparability
condition}. 

In the situation where $\cV = \tanb\cM$, we say that $\mg$ is a \emph{rough metric}
on $\cM$ if in  addition, it is real-valued and the  trivialisations are induced
by coordinate charts.
\end{definition}
 
We distinguish the latter notion of a rough metric from a 
bundle rough metric since the chart induced trivialisations
are necessary to setup a measure.  More precisely,  given a 
rough metric $\mg$ on $\cM$, we obtain 
a global measure $\mu_\mg$ by pasting together  
$$ d\mu_\mg = \sqrt{\det \mg_{ij}}\ d\Leb,$$
via a smooth partition 
of unity subordinate to a covering of $\cM$
by locally comparable charts. This is readily checked
to define a Borel-measure that is finite on compact sets.
We emphasise that a priori, we do not know whether 
$\mg$ induces a length structure. Moreover, a
set $A$ is measurable if and only if it is $\mu_\mg$-measurable.


To emphasise the generality and importance of rough metrics, we give some examples.
Throughout, $\cM$ will denote a smooth manifold.

\begin{example}[Conformally rough metrics]
\label{Ex:Conf}
Let $\mg$ be a smooth (or continuous) metric, and let $f \in \Lp[loc]{\infty}(\cM)$
such that for every compact set $K \subset \cM$, there exists
a constant $\kappa_K$ such that $f \geq \kappa_K$. Then, 
$\mh_f(x) = f(x) \mg(x)$ defines a rough metric.
\end{example} 

Examples of the above type of rough metric were considered in 
\cite{CD} in the context of Ricci flow with degenerate initial metrics.

\begin{example}[Geometry of divergence form operators]
Let $\mg$ be smooth and suppose that $A$ is a real-valued symmetric $(1,1)$-measurable 
bounded, accretive tensorfield, by which we mean: there exist $\kappa > 0$ and $\Lambda < \infty$
such that $\kappa \leq \mg_x(A(x)u, u) \leq \Lambda$ for $x$ almost-everywhere. 
The corresponding divergence form
operator with these coefficients is $\Div_A = -\divv_{\mg} A \conn$ which is the
operator obtained for the symmetric form $J_A[u,v] = \inprod{A \conn u, \conn v}$
with $\dom(J_A) = \SobH{1}(\cM)$. Then, the metric $\mh_x(u,v) = \mg_x(A(x) u,v)$
defines a rough metric and it corresponds to the geometry 
of the operator $\Div_A$. 
\end{example} 

The following examples were first considered in \cite{BLM}.

\begin{example}[Witches hat sphere]
Let $\cM = {\rm S}^n$, the $n$-sphere and $\mh_{R}$ be the round metric.
Let $\phi:B_{1}(p) \to B_{1}^\delta$ be a coordinate chart from 
the ball of radius $1$ near the north pole $p$ to the Euclidean ball of radius $1$.
Inside $B_1^\delta$, define $F: B_1^\delta \to \R^{n+1}$ as
$F(x) = (x, 1  - 2\modulus{x})$ for $x \in B_{1/2}^\delta$
and $F(x) = (x,0)$ for $x \in B_1^\delta \setminus B_{1/2}^\delta$.
The map $F$ is Lipschitz, and its graph on the ball $B_{1/2}^\delta$
is a Euclidean cone.It is easy to see that we can smooth the map $F$ slightly
at $\modulus{x} = 1/2$ to obtain a smooth map $G$ and define
$$ \mg(x) = \begin{cases} \pullb{(G \comp \phi}\delta^n)(x), &x \in B_{1}(p) \\
			\mg_R(x),  &x \not\in B_{1}(p).
	\end{cases}$$
It is clear that this map is smooth everywhere but the north pole $p$
and that $\mg$ is isometric to the Euclidean cone on the ball $B_{1/2}(p)$.
It is readily verified that $(\cM,\mg)$ is a rough metric space.

A generalisation of the above is to replace the sphere with a cylinder of the form 
${\rm S}^n \times (-\infty, 0]$. Then attach a cone, as we did above, to the part ${\rm S}^n \times \{0\}$. This
will produce a non-compact rough metric space.
\end{example}

\begin{example}[Euclidean Box]
Denote the Euclidean box in dimension 
$n$ by 
$${\rm B}^n = \partial\bbrac{-\sqrt{\frac{1}{2(n+1)}}, \sqrt{\frac{1}{2(n+1)}}}^{n+1},$$
and define the radial projection map $G:{\rm B}^n \to \rm{S}^n$ 
by $$G(x) = \frac{x}{\modulus{x}}.$$
A direction calculation via the induced distances show that 
$G$ is a Lipschitz map, and hence, $\mg = \pullb{(G^{-1})}\delta$
is a rough metric on the sphere. Since this is an isometry
between ${\rm B}^n$ with the induced metric
and $({\rm S}^n, \mg)$, we see that the Euclidean box
can be realised as a smooth manifold with a rough metric.
\end{example}

We should also mention that every $C^{k, \alpha}$ metric, for $k \geq 0$ and $\alpha \in [0, 1]$, is rough metric. 
In particular $C^{\alpha}$-limits of Riemannian manifolds, studied by Anderson and Cheeger in 
\cite{AC}, obtained from their precompactness theorem, are rough. 
A general study, analogous to the approach in \cite{AC}, for limits of manifolds with rough metrics does 
not seem possible at this point. The main problem is that there is very little known about the existence of
precompactness theorems for general rough metrics.


\subsection{Global function spaces}

Throughout the remainder of this paper, unless
otherwise specified,  we fix a rough metric
$\mg$ on $\cM$ and a bundle rough metric $\mh$ on $\cV$.
 If $\cV$ and $\mh$ are real,  we consider $\mh$ as a complex-valued inner product
by complexifying $\cV$.
For an open set $U$, define $\Lp{p}(U, \cV)$ spaces as the 
set of measurable distributions $\xi$ such that 
$$ \norm{\xi}^p_{U,p} = \norm{\xi}_{\Lp{p}(U,\cV)}^p = \int_{U} \modulus{\xi(x)}_{\mh(x)}^p\ d\mu_\mg(x) < \infty.$$
When $U = \cM$, we simply write $\Lp{p}(\cV)$
and $\norm{\xi}_p = \norm{\xi}_{\Lp{p}(\cM,\cV)}$.
In the special case of $p = 2$,  the space   $\Lp{2}(U,\cV)$
is a Hilbert space with the inner product 
$$ \inprod{u, v}_{\Lp{2}(U,\cV)} = \int_{U} \mh_x(u(x),v(x))\ d\mu_\mg(x).$$
The $\Lp{2}(\cV)$ inner product will be denoted by $\inprod{\mdot,\mdot}$
and the induced norm by $\norm{\mdot}$.

It can be readily verified
that $\Lp[loc]{1}(\cV)$ is the space of measurable
$\xi$ such that $\xi \in \Lp{1}(K, \cV)$
for all pre-compact $K \Subset \cV$. Furthermore,
the local Sobolev spaces $\Sob[loc]{k,p}(\cV)$
can be characterised by $\xi \in \Lp{p}(\cV)$
such that 
$$ \sum_{i=1}^k \norm{\conn^i \xi}_{K, p} < \infty$$
for every smooth connection $\conn$ over pre-compact 
$K \Subset \cM$. Note that any two smooth connections $\conn_1$ and $\conn_2$
over a pre-compact $K$ are comparable in the sense that  
$$\sum_{j=1}^l \norm{\conn_1^j u}_{K, p} + \norm{u}_{K,p}  
	\simeq \sum_{j=1}^l \norm{\conn_2^j u}_{K, p} + \norm{u}_{K,p}$$
for $l \geq 1$. 

\subsection{Operator theory}

In what is to follow, we will require some notions from 
operator theory. Fixing a Hilbert space $\Hil$,
we consider operators $T: \dom(T) \subset \Hil \to \Hil$, 
where $\dom(T)$ is a subspace of $\Hil$ 
called the \emph{domain} of the operator.
The \emph{range} of $T$ is denoted by $\ran(T)$
and its \emph{null space}, or \emph{kernel}, by $\nul(T)$.

An operator is \emph{densely-defined} if $\dom(T)$
is dense in $\Hil$, and it is \emph{bounded} if
there exists a $C > 0$ such that 
$\norm{Tu} \leq C \norm{u}$. 
We say that an operator is \emph{closed} 
if 
$u_n \to u$ and $Tu_n \to v$ implies
that $u \in \dom(T)$ and $v = Tu$. 
This is equivalent to saying that the
graph $\graph(T) = \set{(u,Tu): u \in \dom(T)}$ is 
a closed subset of $\Hil \times \Hil$.
A subspace $\sA \subset \dom(T)$ is called
a \emph{core} for $T$ if it is dense in 
$\dom(T)$ with respect to the graph 
norm $\norm{\mdot}_{T} = \norm{\mdot} + \norm{T\mdot}$.

An operator is \emph{closable} if $u_n \to 0$ 
and $Tu_n \to v$ implies that $v = 0$.
In that case, $\close{\graph(T)} = \graph(\tilde{T})$
where $\tilde{T}$ is a closed operator.
We write the closure of the operator as $\close{T} = \tilde{T}$.
A densely-defined operator $T$ admits a closed
operator $\adj{T}$ called 
\emph{the adjoint} of $T$ 
with domain
 $$ \dom(\adj{T}) = \inprod{ u\in \Hil: v \mapsto \inprod{Tv, u}\text{ is continuous}}.$$ 
 The operator $\adj{T}$ is defined as follows:
for $u \in \dom(\adj{T})$, there exists $f_u \in \Hil$
such that $\inprod{Tv, u} = \inprod{v, f_u}$ by the Riesz-representation 
theorem, and $\adj{T}u = f_u$.  

An operator $S$ is said to be \emph{an adjoint}
of $T$ if 
$ \inprod{Tu,v} = \inprod{u, Sv}$
for all $u \in \dom(T)$ and $v \in \dom(S)$.
In particular, whenever $T$ and $S$ are densely-defined, 
they are both closable and admit densely-defined adjoints.
An operator $T$ is symmetric if $T$ is \emph{an adjoint}
to itself. However, typically $T \subset \adj{T}$, 
by which we mean that $\dom(T) \subset \dom(\adj{T})$
and $\adj{T} = T$ on $\dom(T)$. An operator
$T$ is \emph{self-adjoint} if $\adj{T} = T$. A
symmetric operator is \emph{essentially self-adjoint}
if it admits a unique self-adjoint extension $T_s$.
In this situation, it is readily checked that
$\close{T} = T_s = \adj{T}$.

\subsection{Main results}
\label{S:MainRes}

Throughout this paper,
whenever we say \emph{first-order differential operator},
we will assume it is an operator
 that takes the form $A^i \partial_i + B$
locally, with coefficients $A^i, B \in \Lp[loc]{\infty}(\adj{\cV} \tensor \cV)$,
 and $A^i \neq 0$ for some $i$. 
This is a linear map  
$\Dir: \Ck{0,1}(\cV) \to \Lp[loc]{\infty}(\cV)$,
that is local and for which   
$\Mul_\eta = \comm{\Dir, \eta \iden}$ is an almost-everywhere  nonzero, fibrewise, bounded multiplication 
operator for $\eta \in \Ck{0,1}(\cM)$. 
 Note that for almost-every $x \in \cM$ 
and each fibre norm $\modulus{\mdot}_x$ on $\cV_x$, 
there exist $C_x > 0$ such that 
$\modulus{\Mul_{\eta}(x)} \leq C_x \modulus{\conn \eta(x)}.$

If  for  $m \geq 0$, $\Dir: \Ck{\infty}(\cV) \to \Ck{m}(\cV)$,
and $\Dir: \Ck{l}(\cV) \to \Ck{l-1}$ for $0 < l \leq m+1$,
then we say that  $\Dir$ has  $\Ck{m}$ coefficients.  

Throughout, let $l \leq m+1$
for a $\Ck{m}$ coefficient operator 
(set $m =0$ if
the operator has $\Lp[loc]{\infty}$ coefficients) 
and  define:  
$$\Sp_l = \Sp_l(\Dir) = \set{u \in \Ck{\infty} \intersect \Lp{2}(\cV): \Dir^l u \in \Lp{2}(\cV)}.$$
Denote the $l$-graph norm of the operator by:
$ \norm{u}_{\Dir^l} = \norm{\Dir^l u} + \norm{u}.$ 
For a function $\eta \in \Ck{\infty}(\cM)$, 
define the commutators
$\Mul^l_\eta u = \comm{\Dir^l, \eta \iden}u,$
for $u \in \Ck{\infty}(\cV)$.  The
commutator $\Mul^1_\eta$ will be denoted by
$\Mul_\eta$. 

We present the following two axioms under which
we prove the most general  results of this paper.
In \S\ref{S:App}, we illustrate that these
axioms are valid for a wide class of elliptic operators.
\begin{enumerate}[({A}1)]
\item
\label{D:First}
\label{D:Close}
whenever $u \in \Sp_1(\Dir)$ and $v \in \Ck[c]{\infty}(\cV)$,
$ \inprod{\Dir u, v} = \inprod{u, \Dir v},$

\item
\label{D:Reg}
\label{D:Last} 
whenever $u \in \Lp{2}(\cV)$ and $\Dir^l u \in \Lp[loc]{2}(\cV)$, then $u \in \Sob[loc]{l,2}(\cV)$.
\end{enumerate}

\begin{remark}
The condition \ref{D:Reg} is an $\Lp{2}$-ellipticity condition 
on the operator $\Dir$. It is automatically satisfied for 
$\Ck{m}$ coefficient elliptic operators. We formulate this
as an axiom since we only require this weaker formulation, 
and as we shall see in \S\ref{S:App}, it can be proved
in instances where the coefficients are merely 
$\Lp[loc]{\infty}$.
\end{remark}

For integers  $1 \leq k \leq m+1$, 
let $\Dir_c^k = \Dir^k$ with domain $\dom(\Dir_c^k) = \Ck[c]{\infty}(\cV)$
and $\Dir_2^k = \Dir^k$ with domain $\dom(\Dir_2^k) = \Sp_k(\Dir).$
For the case $k = 1$, \ref{D:Close} implies that $\Dir_2$ and $\Dir_c$
are closable, and  on letting $\Dir_N = \close{\Dir_2}$ and $\Dir_D = \close{\Dir_c}$,
we obtain that 
$\inprod{\Dir_N u, v} = \inprod{u, \Dir_D v}$
for $u \in \dom(\Dir_N)$ and $v \in \dom(\Dir_D)$. Moreover, using
\ref{D:Reg} and a mollification argument, we obtain that 
$\Ck[c]{0,1}(\cV) \subset \dom(\Dir_D)$. Fix $v \in \Ck[c]{\infty}(\cV)$,
let $V$ be a pre-compact set satisfying
 $\spt v \subset V$, and $\set{\eta_i}_{i=1}^M$ be smooth partition of unity on $V$. 
For $u \in \Sp_k(\Dir)$, write $u = \sum_{i=1}^M \eta_i u$ on $V$. 
Set $u_i = \eta_i u$, extended to zero outside of $\spt \eta_i$.
Since $\Dir$ is a $\Ck{m}$ coefficient operator,
we have that $\Dir^p (u_i) \in \Ck[c]{m - p}(\cV) \subset \Ck[c]{0,1}(\cV)$ for $p \leq k -1$.
Then,
\begin{multline*}
\inprod{\Dir^k u,v} 
	= \sum_{i=1}^M \inprod{\Dir_N^k (u_i), v}
	= \sum_{i=1}^M \inprod{\Dir_N \Dir^{k-1}_N (u_i),v} \\
	= \sum_{i=1}^M \inprod{\Dir^{k-1}_N u_i , \Dir_D v} 
	= \dots
	= \sum_{i=1}^M \inprod{u_i, \Dir_D^k v}
 	= \inprod{u, \Dir^k  v}.
\end{multline*}
That is, 
$$ \inprod{\Dir^k u, v} = \inprod{u, \Dir^k v},\quad \forall u \in \Sp_k(\Dir),\ 
	\forall v \in \Ck[c]{\infty}(\cV).$$
Observe that if we simply asked that $\Dir$ be symmetric on $\Ck[c]{\infty}(\cV)$, 
then a similar calculation would yield \ref{D:Close} as a consequence.

As a consequence of this symmetry condition for $\Dir^k$,
we obtain that the operators $\Dir^k_c$ and $\Dir^k_2$ are closable
and as for the $k=1$ case, we write 
$(\Dir^k)_D$ and $(\Dir^k)_N$ respectively to denote the closures of these operators
with domains $\dom^k_0(\Dir)$ and $\dom^k(\Dir)$.
By viewing $\Dir^k$ distributionally, i.e., for $u \in \Lp{2}(\cV)$
defining $(\Dir^k u)(v) = \inprod{u, \Dir^k v}$ for $v \in \Ck[c]{\infty}(\cV)$,
define the maximal domain as 
$$\dom^k_\infty(\Dir) = \set{u \in \Lp{2}(\cV): \Dir^k u \in \Lp{2}(\cV)}.$$

In addition to the conditions \ref{D:First}-\ref{D:Last}, 
a fundamental criteria  we use and exploit throughout
this paper  is the following. 

\begin{definition}[Negligible Boundary]
\label{Def:NegBdy}
We say that the operator $\Dir^l$ exhibits \emph{negligible boundary} if
\begin{equation*}
\tag{$l$-Neg}
\label{lneg}
\inprod{\Dir^l u, v} = \inprod{u, \Dir^l v}
\end{equation*}
for all $u, v \in \Sp_l(\Dir)$.
\end{definition}

The following theorems we present are phrased 
at the level of generality of rough metrics,
in particular so we can emphasise the regularity 
features that are necessary for obtaining the conclusions. Moreover, 
this allows us to divorce assumptions on 
the coefficients on the operator and the underlying metric.

The first theorem we present is the following. It is proved in \S\ref{S:EssNeg2}. 

\begin{theorem}
\label{Thm:FirstMain}
Let $\Dir^l$ satisfy \ref{D:First}-\ref{D:Last} on a
bundle $\cV$ with a bundle rough metric $\mh$ over
$\cM$ with a rough metric $\mg$. Then, the following are equivalent:
\begin{enumerate}[(i)]
\item $\dom^l_0(\Dir) = \dom^l(\Dir)$, 
\item $(\Dir^l)_D$ is self-adjoint, 
\item $(\Dir^l)_N$ is self-adjoint,
\item $\Dir^l$ satisfies \eqref{lneg},
\item $\Dir_c^l$ is essentially self-adjoint.
\end{enumerate}
\end{theorem}

The next theorem allows us to deduce essential
self-adjointness of lower powers of the operator
when we know this for higher powers. 
It is proved in \S\ref{S:core}. 

\begin{theorem}
\label{Thm:FirstMain2}
Let $\Dir$ satisfy \ref{D:First}-\ref{D:Last} on a
bundle $\cV$ with a bundle rough metric $\mh$ over
$\cM$ with a rough metric $\mg$.  
If $\Dir^l$ on $\Ck[c]{\infty}(\cV)$ is essentially self-adjoint, then 
$\Ck[c]{\infty}(\cV)$ is a core for $\modulus{\Dir_D}^\alpha$
for $\alpha \in [0, k]$ and moreover, it is a
core for $\Dir_D^k$ for $k = 1, \dots, l$.
\end{theorem}

\begin{remark} 
This result requires very little geometric properties:
the theorem is phrased for rough metrics which may not, 
a priori, even induce a length structure.
In particular, it does not require completeness.
  
Moreover, note that we only require
the first power of $\Dir$ to satisfy \ref{D:First}-\ref{D:Last}
in this theorem. 
\end{remark}

The second theorem we present is in the case 
that the rough metric induces a length structure. 
By this, we mean that  
$$\met(x,y) = \inf \set{\len(\gamma): \gamma \text{ is an absolutely continuous curve between $x$ and $y$}},$$
where $\len(\gamma)= \int_{I} \modulus{\dot{\gamma}(x)}_{\mg(x)}$,
is a distance metric which yields
a topology that agrees with the topology on $\cM$. 
The quintessential nontrivial example of such a situation is the pullback metric
$\mg = \pullb{\phi}\mgt$, where $\mgt$ is a smooth metric and
$\phi$ is a lipeomorphism. The metric $\mg$, in general, only 
has measurable coefficients, and $\phi$ can be seen as a bi-Lipschitz
transformation of $\mgt$.

\begin{theorem}
\label{Thm:SecondMain}
Let $\cV$ be a vector bundle with a bundle rough metric $\mh$, over
a manifold $\cM$ with a rough metric $\mg$ that induces
a complete length space.
If $\Dir$ is a first-order operator satisfying \ref{D:First}-\ref{D:Last},
then $\Dir$ on $\Ck[c]{\infty}(\cV)$ is essentially self-adjoint.
If $\Dir$ has $\Ck{m}$ coefficients
and $\Dir^k$ satisfies \ref{D:First}-\ref{D:Last}
for $1 \leq k \leq m+1$, then  
$\Dir^k$ on $\Ck[c]{\infty}(\cV)$ is essentially
self-adjoint for $1 \leq k \leq m+1$.
\end{theorem} 

The proof of this theorem can be found in \S\ref{S:Ess-k}. 

\begin{remark}
In \cite{Wolf}, Wolf obtains essential
self-adjointness for  $\Dir$ and $\Dir^2$
when $\Dir$ is the Dirac operator for \emph{smooth} 
complete metrics. Following from this, Chernoff in \cite{Chernoff} 
uses wave techniques to obtain such a result for 
all powers $\Dir^k$ of first-order differential operators
satisfying a velocity of propagation condition.  
 The  wave techniques crucially rely on smoothness. 

In our theorem, we only assume completeness and 
allow for
singularities at the level of the metric. 
 We  keep  
track of the way in which the coefficients of our
differential operator affects the density of $\Ck[c]{\infty}(\cV)$
for higher powers.  
\end{remark}
\section{Applications}
\label{S:App}

\subsection{Dirac operators on Dirac bundles}
\label{Sec:App:Dirac}

Fix a  $\Ck{0,1}$ metric $\mg$ on $\cM$.
By the usual formula for the Christoffel symbols, 
we define the 
Levi-Civita connection 
$\conn:\Ck{0,1}(\tanb\cM) \to \Lp[loc]{\infty}(\cotanb \cM\tensor \tanb\cM)$. 

Let $\Cliff\cM$ denote the Clifford algebra,
which is the exterior algebra $\Forms\cM$
equipped with the Clifford product $\cliff$
given by 
$ \alpha \cliff \beta = \alpha \wedge \beta + \alpha  \cut \beta.$
By the symbol $\conn$, we denote the canonical 
lift of the Levi-Civita connection to the bundle
$\Cliff\cM$.

Following the notation of \cite{ML}, we say that
a complex bundle $\cS$ is a Spin bundle over $(\cM,\mg)$
if:
\begin{enumerate}[(i)]
\item it is a bundle of left-modules over $\Cliff\cM$
	with $\mdot: \Cliff\cM \times \cS \to \cS$,
\item it is equipped with a hermitian $\Ck{0,1}$-metric 
and a compatible connection $\conn:\Ck{0,1}(\cS) \to \Lp[loc]{\infty}(\cotanb\cM \tensor \cS)$ 
satisfying:
\begin{equation}
\label{App:D}
\tag{S}
\mh_x(e \rep \sigma,  e \rep \gamma) = \mh_x (\sigma, \gamma),
\ \text{and}\ 
\conn_v (\eta \rep u) = (\conn_v \eta) \rep u + \eta \rep \conn_v \psi
\end{equation}
for all $\modulus{e} = 1$ in $\cotanb_x\cM$, $v \in \tanb\cM$
and  for almost-every $x \in \cM$. 
\end{enumerate}  
The first property says that $\rep$ is unitary 
and the second property says that $\conn$
is a module derivation. 

Fixing a frame $\set{e_i}$ for $\tanb\cM$, recall that 
the Dirac operator $\Dir$ is defined by 
$$\Dir \sigma = e^i \rep \conn[e_i]\sigma.$$
It is easy to check that this is a first-order
differential operator, $\Dir:\Ck{0,1}(\cS) \to \Lp[loc]{\infty}(\cS)$
as a consequence of the module derivation property 
in \eqref{App:D}.

We now prove that $\Dir$ satisfies \ref{D:Close}. But first,
we present the following lemma that will be the primary tool
used  in its proof. 

\begin{lemma}
\label{Lem:DivLip}
Let $f \in \Ck[c]{0,1}(\cotanb\cM )$, a compactly supported Lipschitz co-vectorfield.
Then,
$$ \int_{\cM} \divv_\mg f\ d\mu_\mg = 0.$$
\end{lemma}
\begin{proof}
Since the metric $\mg$ is at least continuous, we note as in 
Proposition 13 of \cite{BRough} that, given any $C > 1$,
there is a smooth metric $\mgt$ which is $C$-close to $\mg$, 
by which we mean that:
$ C^{-1}\modulus{u}_{\mgt} \leq \modulus{u}_{\mg} \leq C \modulus{u}_{\mgt}$
for every $x \in \cM$ and $u \in \tanb_x\cM$.

Note that there exists an $E \in \Ck{0,1}(\Tensors[1,1]\cM)$ symmetric such that 
$\mg(u,v) = \mgt(Eu,v)$, and on letting $\theta = \sqrt{\det E}$, 
we have that $\mu_\mg = \theta \mu_{\mgt}$.
Also, $E$ is bounded with respect to both metrics
$\mg$ and $\mgt$. Computing in $\Lp{2}$ allows us to assert
that $\divv_\mg f = \theta^{-1} \divv_{\mgt}( \theta E f)$
(c.f. Proposition 12 in  \cite{BRough}).
Therefore, 
$$\int_{\cM} \divv_\mg f\ d\mu_\mg  
	= \int_{\cM} \theta^{-1} \divv_{\mgt}(\theta E f)\ d\mu_\mg
	= \int_{\cM} \divv_{\mgt}(\theta E f)\ d\mu_{\mgt}.$$
But since the metric $\mgt$ is smooth, $\theta E \in \Ck{0,1}(\Tensors[1,1]\cM)$, 
$f \in \Ck[c]{0,1}(\cotanb\cM)$, and therefore $\theta E f \in \Ck[c]{0,1}(\cotanb\cM)$. 
Thus, by Lemma 7.113 in \cite{RF}, we obtain that
$\int_{\cM} \divv_{\mgt} (\theta E f)\ d\mu_{\mgt} = 0$.
\end{proof}

With this, we prove the following.
\begin{proposition}
Whenever $\sigma,\ \gamma \in \Ck[c]{\infty}(\cS)$, we have that
$\inprod{\Dir \sigma, \gamma} = \inprod{\sigma, \Dir \gamma}$.
Moreover, $\Dir$ satisfies \ref{D:Close}.
\end{proposition}
\begin{proof}
Fix a point $x \in \cM$ where $\mg$ and $\mh$ are differentiable
and fix a smooth frame $\set{e_j}$ near $x$.
Let $\conn[e_j] e_k = C_{jk}^m e_m$ and note that
by the lifting of $\conn$ to $\cotanb\cM$, we have that
$\conn[e_j]e^k = - C_{ji}^k e^i.$ 
Using metric compatibility, we obtain 
\begin{align*}
\mh(\Dir \sigma, \gamma)
	= \mh(e^j \rep \conn[e_j] \sigma, \gamma) 
	&= - \mh(\conn[e_j] \sigma, e^j \rep \gamma) \\
	&= \mh(\sigma, \conn[e_j] (e^j \rep \gamma)) - e_j\mh(\sigma, e^j \rep \gamma)  \\
	&= -C_{jm}^k \mh(\sigma, e^m \rep \gamma) \delta^j_k - e_j\mh(\sigma, e^j \rep \gamma) + \mh(\sigma, \Dir \gamma).
\end{align*}

Now, note that for $w = w_i e^i \in \Ck{0,1}(\tanb\cM)$,
$$-\divv w = \tr\conn w = e_j w^j + w^j \tr(e^j \tensor \conn[e_j]e_i)
	= e_j w^j + w^i C_{ij}^k \delta^j_k.$$
Define the vectorfield $W_{\sigma,\gamma}(x) = \mh_x(\sigma(x), e^k(x) \gamma(x)) e_k(x)$,
extended to the whole of $\cM$ by zero outside of $\spt \sigma \intersect \spt \gamma$. 
It is easy to see that this is invariantly defined and that 
it is $\Ck{0,1}$. Thus, from combining these calculations, we see that
$\mh(\Dir\sigma, \gamma) = \divv W_{\sigma,\gamma} + \mh(\sigma, \Dir \gamma)$
at points of differentiability.
Thus, by invoking Lemma \ref{Lem:DivLip}, we obtain that
$$ \int_{\cM} \divv W_{\sigma,\gamma}\ d\mu_\mg = 0,$$
and therefore, $\Dir$ is symmetric on $\Ck[c]{\infty}(\cS)$.
The proof of \ref{D:Close} follows from this as noted
in \S\ref{S:MainRes}.  
\end{proof} 
\begin{remark}
Observe that unlike the smooth case, we cannot assume 
that we can solve for a frame in which $(\conn[e_j]e_i)\rest{x} = 0$,
since the  metrics $\mg$ and $\mh$  are merely $\Ck{0,1}$. 
\end{remark}

Next, fix $v \in \tanb_x\cM$ and $u \in \cS_x$.
Let $\eta: \cM \to \R$ be a compactly supported
smooth function such that 
$\extd \eta(x) = v$. Since 
$\Dir$ is a first-order differential operator, 
the symbol of $\Dir$ at $(x,v)$ is: 
$$\Sym_{\Dir}(x,v)u 
	= \Mul(\eta)u 
	= \Dir(\eta u) - \eta \Dir u
	= \extd \eta \rep u
	= v \rep u.$$
 The critical regularity case
in the following theorem is
modelled on 
the proof of Proposition 3.4 in \cite{BMcR}.

\begin{proposition}
\label{Prop:DiracA2}
$\Dir$ satisfies \ref{D:Reg}. If the metrics $\mg$
and $\mh$ are $\Ck{m}$ for $m \geq 0$
and $\Dir$ is a $\Ck{m}$ coefficient operator, 
\ref{D:Reg} holds for $\Dir^l$ for $l \leq m + 1$.
\end{proposition}
\begin{proof}
First, observe that any two \emph{smooth} connections
$\conn_1$ and $\conn_2$ are comparable on precompact
subsets in $\Lp{2}$-norm. 
Hence, to show \ref{D:Reg}, it suffices to show that 
$\Dir \sigma \in \Lp{2}(\cS)$ implies that $\sigma \in \Sob{1,2}(U, \cS)$
inside  precompact trivialisations $(U, \phi)$
corresponding to charts.
Also note that $\Dir u \in \Lp{2}(\cS)$ means exactly 
that
$\modulus{(\Dir \sigma)(\gamma)} = \modulus{\inprod{\sigma, \Dir \gamma}} \lesssim \norm{\gamma}$
for all  $\gamma \in \Ck[c]{\infty}(\cS)$.

Suppose that $\sigma \in \Lp{2}(\cS)$ with 
$\spt \sigma \subset U$ compact. Fix $\gamma \in \Ck[c]{\infty}(\cS)$
with $\spt \gamma \subset U$. Since we are inside a precompact chart
and $\mg$ is $\Ck{0,1}$, we obtain a constant $C \geq 1$
such that 
$$C^{-1}\modulus{u}_{\mg} \leq \modulus{u}_{\pullb{\phi}\delta}
	\leq C \modulus{u}_{\mg}$$
for all $x \in U$, where by $\delta$ we denote the
Euclidean metric. Moreover, there exists
a transformation $\B \in \Ck{0,1}(\Tensors[1,1]\cM)$
such that $\mg(u,v) = \pullb{\phi}\delta(\B u, v)$
and so that $d\mu_\mg = \uptheta d\pullb{\phi}\Leb$, 
where $\uptheta = \sqrt{\det \B}$ and $\pullb{\phi}\Leb$
is the pullback of the Lebesgue measure in $\phi(U)$
by $\phi$. Then, 
we obtain by the fact that $\Dir \sigma \in \Lp{2}(\cS)$
that $\modulus{\inprod{\sigma, \Dir \gamma}} \lesssim \norm{\gamma}$. 
Expanding this norm, taking an orthonormal frame $\set{e_i}$
for $\tanb\cM$ (note that this frame is Lipschitz) 
and $\set{s_j}$ for $\cS$, we have that 
$$ 
\modulus{\int_{U} \mh{(\sigma, e^i \rep (\conn[e_i] \gamma^k) s_k) }\ \uptheta d\Leb } \lesssim \norm{\gamma} $$
since $C_{ij}^k = e^k(\conn[e_i]s_j) \in \Lp{\infty}(U, \cS)$. 
Furthermore, $\conn[e_i]\gamma^k \uptheta = \conn[e_i](\theta \gamma^k) - (\conn[e_i]\theta)\gamma^k$, and
so again,  we have that 
$$
\modulus{\int_{U} \mh{\sigma, e^i \rep (\conn[e_i] \gamma^k) s_k) )}\ d\Leb } 
	\lesssim \norm{\gamma}_{\Lp{2}(U,\cS,\Leb)}.$$
 Since we fixed an orthonormal frame $s_k$,
the action $e^i \rep s_k = A(e^i) s_k$ for a matrix $A(e^i)$,
where the coefficients of  $A(e^i)$ are constant inside $U$.   
Viewing this integral on $\R^n$ and extending $\sigma$
to the whole of $\R^n$ by setting it to $0$ outside of
$\phi(U)$, we obtain for any $\gamma \in \Ck[c]{\infty}(\R^n, \C^N)$
(where $\C^N \cong \cS$ inside $U$),
$\modulus{\inprod{\sigma, \conn[e_i](\uptheta \gamma^k) e^i \rep s_k}} \lesssim \norm{\gamma}$.
Since the matrix $A(e^i)$ is constant coefficient, 
we have that $\widehat{e^i \rep \psi} = e^i \rep \widehat{\psi}$
for any $\psi \in \Lp{2}(\R^n,\C^N)$, where $\ \widehat{}\ $ denotes
the Fourier Transform.  Then,
$$
\inprod{\sigma, \conn[e_i](\uptheta \gamma^k) e^i \rep s_k} 
	= \inprod{\hat{\sigma}, e^i \rep \xi_i \widehat{ (\uptheta \gamma^k)}}.$$
This is valid for any $\gamma \in \Ck[c]{\infty}(\R^n, \C^N)$, therefore, 
$\xi \rep \widehat{\sigma} = \xi_i e^i \rep \widehat{\sigma} \in \Lp{2}(\R^n,\C^N)$.
Hence, $\sigma \in \Sob{1,2}(\R^n, \C^N)$ 
which implies that $\sigma \in \Sob[loc]{1,2}(U, \cS)$.

For a general $\sigma \in \Lp{2}(\cS)$ with $\Dir \sigma \in \Lp{2}(\cS)$, fix $\tilde{U}$ a
precompact chart and $U \Subset \tilde{U}$. 
There is a smooth partition of unity $\set{\rho_j}$ inside $\tilde{U}$
such that $\eta = \sum_{j=1}^M \rho_j = 1$ on $U$, and
we extend $\eta = 0$ outside $\tilde{U}$. Then,
for any $\gamma \in \Ck[c]{\infty}(\cS)$, 
$$\modulus{\Dir(\eta \sigma)(\gamma)} = \inprod{\sigma, \eta \Dir \gamma} 
	= \inprod{\sigma, \Dir(\eta \gamma)} - \inprod{\sigma, \extd \eta \rep \gamma}.$$
 Since $\Dir \sigma \in \Lp{2}(\cS)$ we have 
$\modulus{\inprod{\sigma, \Dir(\eta \gamma)}} \lesssim \norm{\eta \gamma} \leq \norm{\gamma}$
and it is easy to see, from the Cauchy-Schwarz inequality, that 
$\inprod{\sigma, \extd \eta \rep \gamma} \lesssim \norm{\sigma} \norm{\gamma}$
(where the constant contains $\sup \modulus{\extd \eta} \lesssim 1$). 
This proves that $\Dir(\eta \sigma) \in \Lp{2}(\cS)$, and since  
$\norm{\Dir \sigma}_{\Lp{2}(U, \cS)} \leq \norm{\Dir(\eta \sigma)}$, 
by running our previous argument with 
$\eta \sigma$ in place of $\sigma$, 
we obtain that $\sigma \in \Sob[loc]{1,2}(\cS)$ whenever $\Dir \sigma \in \Lp{2}(\cS)$.
The estimate in \ref{D:Reg} is immediate. This proves the critical regularity 
case.

For the case that the coefficients of the metric and
operator are $\Ck{m}$, and since 
$\Sym_{\Dir}(x,v)\sigma  = v \rep \sigma$
implies that $\Dir$ is an elliptic operator, 
\ref{D:Reg} follows from the classical Schauder interior regularity estimate. 
See Theorem 10.3 in \cite{Vargas}.
\end{proof}



%

\subsection{Application to Dirac type operators}


The study of Lipschitz metrics on smooth manifolds has, in recent decades, received a lot of attention
by various authors. One such study came from work of Chen and Hsu in \cite{CH}.
In that paper, Chen and Hsu considered a smooth complete Riemannian manifold with a Lipschitz metric, and
took up the study of gradients of certain solutions to Laplace's equation in this setting. As the metric
was only assumed to be Lipschitz, they looked at solutions of $\Sob[loc]{1,2}$ regularity, which they
then termed weakly harmonic. An assumption on the volume growth of geodesic balls, and a bound on the 
(distributional) Ricci curvature, led them to
obtain gradient estimates, generalising those of Yau in the smooth setting, for such weakly harmonic
functions. Their work represents a typical situation where one has to consider a differential operator in a setting
where the metric is not smooth, and where one wants to understand certain qualitative properties of solutions
of a PDE associated to the differential operator.



A natural question that arises from their work is, 
how properties, like being essentially self-adjoint, of a differential operator are affected in the Lipschitz 
setting. In this section we will give an application of our work to the case of the Hodge Dirac and
Atiyah Singer Dirac operators.



We will start with a general theorem about the Dirac operator, the notation being that which was used in
the previous section.

On combining the propositions from the previous section, we use
Theorem \ref{Thm:FirstMain} and  Theorem \ref{Thm:SecondMain}
to obtain the following theorem.
The metric $\mg$ appearing in this theorem
automatically induces a length space
as a consequence
of Proposition 4.1 in \cite{Burtscher} 
by Burtscher.

\begin{theorem}
The operator $\Dir$ on $\Ck[c]{\infty}(\cS)$ is essentially self-adjoint.
If the metric $\mg$ 
and $\mh$ are of class $\Ck{m}$ for $m \geq 0$, with $\mg$ complete, and $\Dir$ is a $\Ck{m}$
coefficient operator, then $\Dir^l$ on $\Ck[c]{\infty}(\cS)$ is 
essentially self-adjoint for $1 \leq l \leq m+1$. 
Moreover, $\Ck[c]{\infty}(\cS)$ is a core for $\modulus{\Dir_D}^\beta$ for $\beta \in [0, m+1]$.
\end{theorem}

We list two noteworthy consequences of this theorem - to the Hodge Dirac
operator and to the Atiyah Singer Dirac operator.
 We remind the reader that in the following corollaries,
a $\Ck{m}$ coefficient metric induces a $\Ck{m-1}$
coefficient connection, from which the
respective operators are built. 


\begin{corollary}
Let $\mg$ be a $\Ck{0,1}$ complete metric
on a smooth manifold $\cM$. Then,
the Hodge Dirac operator 
$\Dir = \extd + \intd$ on $\Ck[c]{\infty}(\Forms\cM)$, 
 where $\Forms\cM$ is the exterior-algebra, 
is essentially self-adjoint. 
If $\mg$ is $\Ck{m}$, $m \geq 1$,  then we get that
$\Dir^l$ on $\Ck[c]{\infty}(\Forms\cM)$ is 
essentially self-adjoint for $1 \leq l \leq m$.
\end{corollary}

\begin{remark}
Gaffney in \cite{Gaffney} obtains this result for $\Dir^2$
on $\Ck{m}$ manifolds. Roelcke in \cite{Roelcke} obtains this theorem for $\Dir^2$ restricted
to functions, and Cordes \cite{Cordes} obtains this for $\Dir^m$ restricted
to functions when the metric is smooth. Gaffney 
in \cite{Gaffney} obtains it for all powers under
the smoothness assumption.

\end{remark}




\begin{corollary}[Theorem \ref{Thm:Main2}]
Let $\mg$ be a $\Ck{0,1}$ complete metric
on a smooth Spin manifold $\cM$, with a spin
structure $\Prin{\Spin}(\cM)$. Let 
$\Spinors \cM = \Prin{\Spin} \times_\eta \Spinors\R^n$, 
where $\eta:\Spin_n \to \bddlf(\Spinors \R^n)$ is the 
usual complex
irreducible representation (given by the nontrivial minimal irreducible complex representation in odd dimensions 
or the sum of the two nontrivial minimal irreducible complex representations in even dimensions). 
Then, the associated \emph{Atiyah Singer Dirac 
operator} $\spin{\Dir}$ is essentially self-adjoint on $\Ck[c]{\infty}(\cV)$.
If the metric is $\Ck{m}$ for $m \geq 1$, then $\spin{\Dir}^l$
is essentially self-adjoint on $\Ck[c]{\infty}(\cV)$ for $l \leq m$.
\end{corollary}

\begin{remark}
When the metric $\mg$ is smooth, Wolf in \cite{Wolf} 
obtains this for $\spin{\Dir}$ and $\spin{\Dir}^2$, and
Chernoff in \cite{Chernoff} obtains this for all
powers $\spin{\Dir}^m$.
The first part of this theorem, was also obtained
by Bandara, McIntosh and \Rosen\ in \cite{BMcR}
\end{remark}

\subsection{Application to Elliptic operators with $\Ck{m}$ coefficients}


In the previous section, we focused our attention to manifolds admitting metrics of Lipschitz regularity.
The study of more general non-smooth spaces has also received a lot of attention in the last few decades. 
Let us survey a few of these works.

In \cite{AC}, Anderson and Cheeger studied smooth manifolds admitting
$\Ck{\alpha}$-metrics. Through their precompactness theorem they were able to show that such singular spaces
can be seen to arise as limits of smooth manifolds with smooth metrics, admitting bounds on their curvature,
injectivity radius, and volume. Thus, in a certain $\Ck{\alpha}$-topology, these singular spaces were arising
as limit points of smooth spaces.

In \cite{CD}, Chen and Ding study the Ricci flow with degenerate initial metric.
That is, they study metrics of the form $\e^{u_0}g_0$, where $g_0$ is smooth and $\e^{u_0} \in \Lp{\infty}$
(this example falls into the category of conformally rough metrics, defined earlier in  Example \ref{Ex:Conf}). 
They were able to show that the Ricci flow,
with this conformally rough metric as an initial metric, had a solution which, for positive time, was smooth.

Rough metrics have also recently made their way into the study of certain problems motivated from
Physics. In \cite{GrantTassotti}, Grant and Tassotti show that 
the positive mass theorem holds for continuous Riemannian metrics with $\Sob[loc]{2,n/2}$ regularity on
manifolds of dimension less than or equal to 7 or spin manifolds in all dimensions. In doing so they generalised 
work of Schoen and Yau \cite{SchoenYau}, and Witten \cite{Witten}, in the smooth setting, to the rough setting. 

Another recent work, applying the theory of rough metrics, was the work of Bandara, Lakzian, and Munn in \cite{BLM}. 
In this work, the authors study a geometric flow,
introduced by Gigli and Mantegazza (see \cite{GM}), in the context of smooth manifolds
with rough metrics, and sufficiently regular heat kernels. They are able to provide a regularity theory for the flow
on such singular spaces in terms of the regularity of the heat kernel.  

There are many other works that take up the study of smooth manifolds admitting non-smooth metrics, but the 
understanding of differential operators and their qualitative behaviour, on such spaces, still remains a largely
unexplored topic. For example, many of the references outlined so far are really looking at the study of certain
PDEs on such singular spaces. They seek to understand features about these PDEs that are well known in the
smooth setting. With this as motivation, we can take a first-order elliptic operator, $\Dir$, or a power of such 
an operator, and seek to understand solutions to heat type equations, $u_t + \Dir u = 0$, or 
Schrodinger type equations, $u_t + i\Dir u = 0$, on the singular spaces used in the above references. 
From a functional analytic point of view, the way to approach such a question is to build the appropriate 
propagator for the equation, and this in turn leads to the question of essential self-adjointness of the operator
in the non-smooth setting.

Motivated by the above, in this section we will give an application of our work to the case of elliptic differential 
operators in the case of non-smooth metrics, which covers all those works mentioned above, and in the introduction.
We have decided to state the theorem in the case of general rough metrics. The reader who is not comfortable with
this level of generality can pick their favourite rough metric, and see the theorem as a statement about
the study of such elliptic operators on such non-smooth spaces.


In the following theorem, the differential operators $\Dir$ are with $\Ck{m}$ coefficients. 
As a consequence of Theorem \ref{Thm:SecondMain}, we obtain the following.

\begin{theorem}
Let $\cV$ be a smooth bundle over a smooth manifold $\cM$ with 
rough metrics $\mh$ and $\mg$. Suppose that $\mg$ induces a complete
length space. Let $\Dir$ be a first-order elliptic 
differential operator with $\Ck{m}$ coefficients, $m \geq 0$,
that is symmetric on $\Ck[c]{\infty}(\cV)$ in $\Lp{2}(\cV)$.
Then, $\Dir^k$ on $\Ck[c]{\infty}(\cV)$ for $k = 1, \dots, m + 1$ is essentially self-adjoint. 
\end{theorem}
\begin{proof}
Fix $v \in \Ck[c]{\infty}(\cV)$ and  $u \in \Sp_1(\cV)$.
By a partition of unity argument, we can find $\phi \in \Ck[c]{\infty}(\cM)$
such that $\phi = 1$ on $\spt v$. Thus, on $\spt v$, $\phi u = u$, and
hence, 
$$\inprod{\Dir u, v} = \inprod{\Dir (\phi u), v} = \inprod{\phi u, v} = \inprod{u, \Dir v},$$
where the second equality follows from the symmetry of $\Dir$ on $\Ck[c]{\infty}(\cV)$.
Thus, \ref{D:Close} is satisfied. The condition \ref{D:Reg} 
is a consequence of elliptic regularity, see Theorem 10.3 in \cite{Vargas}.
\end{proof}

\begin{remark}  
This theorem is a generalisation of 
\cite{Chernoff} by Chernoff for elliptic operators
and with minimal regularity assumptions on the underlying
metrics. Moreover, Chernoff assumes
that $\int_{0}^\infty c(r)^{-1}\ dr = +\infty$,
where $c(r)$ is the ``local velocity of propagation'' for
$\Dir$ inside a ball of radius $r > 0$. In the elliptic context, 
we are able to show that such an assumption is not necessary. 
\end{remark}


\subsection{A remark on the smooth setting: Chernoff's velocity of propagation condition}

In this section our smooth manifold and smooth bundle will always have smooth metrics.

One of the significant advances in the study of the essential self-adjointness of first-order
differential operators, and their powers, in the smooth setting, was made by Chernoff in \cite{Chernoff}.
In that paper, Chernoff studies essential self-adjointness via certain hyperbolic systems. 
Under certain conditions, he is
able to prove that a global smooth solution of such a system exists for all time 
(see p. 407 in \cite{Chernoff}), which he is then able to use to study essential self-adjointness.
The main condition that Chernoff assumes is to do with the local velocity of propagation of his operators
(see p. 407 in \cite{Chernoff}).

A consequence of our work is that for operators satisfying \ref{D:First} and \ref{D:Last}
(e.g. symmetric elliptic operators) in the smooth setting, Chernoff's local velocity of 
propagation condition is not necessary to obtain essential self-adjointness.

The  \emph{local velocity of propagation} is defined by:  
\begin{equation*}
c(x) = \sup\{\modulus{\comm{\Dir, f}(x)}_{\mathrm{op}} : f \in \Ck{\infty}(\cM), \vert \conn(f) (x)\vert_{\mg} = 1\},
\end{equation*}
and the \emph{velocity of propagation inside a ball} by
$c(r) = \sup\set{c(x) : x \in B_r}.$

The condition Chernoff imposes on his operator takes the form of a divergent integral. Namely, he requires
that 
$$ \int_{0}^\infty \frac{1}{c(r)}\ dr = +\infty. $$  

Although many operators in applications do satisfy this divergent integral condition, we would like to point out
that there are many others that do not. For these operators, the techniques of Chernoff, even in the setting
of smooth metrics, is inadequate to prove essential self-adjointness.

It is not so hard to construct examples of first order operators that do not satisfy the above condition. 
Let us give one. Consider a smooth, real-valued nonzero function $f$, 
let  $c_f(r) = \sup\set{f^2(x): x \in B_r}$, and assume that $c_f^{-1} \in \Lp{1}([0,\infty))$.
Let $\Dir$ denote a Dirac operator on a Dirac bundle $\cS$ with smooth coefficients. 
From the symbol computation carried out in Section \ref{Sec:App:Dirac}, it is easy to see that for such an
operator we have $c^{\Dir} (r) = 1$. Let us now consider the operator, $\Dir_f u= f\Dir (f u)$, 
which is a symmetric operator on $\Ck{\infty}(\cS)$.
It is readily verified that the principal symbol $\comm{\Dir_f,\eta\iden}(x) = f^2(x) \comm{\Dir, \eta\iden}(x)$
(where $\eta$ is a smooth function)
and the speed of propagation for the operator $\Dir_f$ is $c^{\Dir_f}(r) = c_f(r)$.
Therefore $\Dir_f$ is again elliptic as $f \neq 0$,
and since we assumed $c_f^{-1} \in \Lp{1}([0,\infty))$, 
the integral 
$$\int_{0}^\infty \frac{1}{c^{\Dir_f}(r)}\ dr = \int_{0}^\infty \frac{1}{c_f(r)}\ dr < \infty.$$
Thus Chernoff's condition does not hold for the  operator $\Dir_f$. 
However, by
applying Theorem \ref{Thm:FirstMain2}
to this operator, we obtain that $\Dir_f$ and its powers are all essentially self-adjoint.
Note that a similar conclusion follows if we replace $\Dir_f$ with $\Dir_f + g \iden$, where
$g \in \Lp[loc]{\infty}(\cM)$ and real-valued.

We would like to point out to the reader that Chernoff's work is for general symmetric first-order 
differential operators, and their powers. He makes no assumption on the regularity of the 
operator, whereas in our case, we assume regularity via \ref{D:Reg}. Therefore, while our methods are 
robust enough to handle many operators that typically arise in applications 
that do not satisfy Chernoff's assumptions, 
they are certainly not adequate for the study of more general operators that do not satisfy \ref{D:Reg}.


\section{Essential self-adjointness and the negligible boundary condition}
\label{S:EssNeg}
\label{S:EssNeg2}

In this section, we establish the equivalence of negligible boundary 
and essential self-adjointness for powers of operators. This is the crucial property we 
exploit throughout the later parts of this paper. We emphasise
that the background geometric assumptions here are minimal. 
In particular, we do not assume completeness.

The central tool is for us to be able to 
equate $\dom^l(\Dir)$ with $\dom_\infty^l(\Dir)$.
We establish some preliminary lemmas that will aid
us  in obtaining this equality.   

\begin{lemma}
\label{Lem:Local}
Let  $u \in \Sob[loc]{l,2}(\cV)$ 
with $\spt u$ compact,
$U$ a precompact open set with $\spt u \subset U$
and $\conn^U$ any smooth connection in $U$.
Then, $u \in \dom^l_0(\Dir)$ and there exists a constant $C > 0$
(dependent on $\conn^U$ and $U$) such
that
$$\norm{\Dir^l u}_{\Lp{2}(U,\cV)} \leq 
		C \cbrac{\sum_{i=1}^l \norm{(\nabla^U)^{i} u}_{\Lp{2}(U, \cV)} + \norm{u}_{\Lp{2}(U, \cV)}}.$$ 
\end{lemma}
\begin{proof}
Observe that, via a partition of unity argument, 
we can assume that $\spt u \subset U$ with $U$
corresponding to a chart. Also, since $U$ is precompact, 
any two smooth connections, as well as all of their powers, 
are comparable. Thus, we assume that $\conn^U$ is the flat
connection inside $U$ with respect to a fixed trivialisation.

Next, note by the hypothesis on $u$, we have that
$u \in \Sob{l,2}(U,\cV)$, where by what we have said
we can fix the norm 
$\norm{u}_{\Sob{1,2}(U,\cV)} = \sum_{j=1}^l \norm{(\conn^U)^j u} + \norm{u}.$
Thus, there is a sequence $u_n \in \Ck[c]{\infty}(\cV)$ such that
$u_n \to u$ in $\norm{\mdot}_{\Sob{1,2}(U,\cV)}$.
Using the fact that $\Dir = A^i \partial_i + B$ inside $U$, 
we obtain  
$$\norm{\Dir^l (u_n - u_m) } \lesssim \sum_{j=1}^l \norm{(\conn^U)^j (u_n - u_m)}_{U} + \norm{u_n - u_m}_{U}.$$
As $m,\ n \to \infty$, the right hand side of this expression tends to zero, 
and hence, we obtain that $\Dir^l u_n \to v$. Since $\Dir^l$ is closable
and the sequence $u_n \in \Ck[c]{\infty}(\cV)$,
we conclude that $u \in \dom^l_0(\Dir)$ and $v = (\Dir^l)_D u$.
The estimate follows easily.  
\end{proof}

With the aid of this lemma, we compute the maximal domain. 
Our proof here is inspired by the 
work of Masamune in \cite{Masamune} (see Theorem 2, p.114).

\begin{proposition}
\label{Prop:DomMax}
Under the assumptions \ref{D:First}-\ref{D:Last}  
for $\Dir^l$,  we have the equality $\dom^l(\Dir) = \dom^l_\infty(\Dir) = \dom( \adj{(\Dir_c^l)})$.
\end{proposition}
\begin{proof}
We first prove the second equality. For that, note that: 
\begin{align*} 
\dom^l_\infty(\Dir) 
	&= \set{u \in \Lp{2}(\cV): \Dir^l u \in \Lp{2}(\cV)} \\
	&= \set{u \in \Lp{2}(\cV): \modulus{(\Dir^l u)(v)} = \modulus{ \inprod{u, \Dir^l v}} \lesssim \norm{v},\ \forall v \in \Ck[c]{\infty}} \\ 
	&= \set{u \in \Lp{2}(\cV): v \mapsto \inprod{u,\Dir^l v}\ \text{is continuous}} = \dom( \adj{(\Dir_c^l)}).
\end{align*}

To prove the first equality, it suffices to
show that $\dom^l_\infty(\Dir) \subset \dom^l(\Dir)$. 
Fix $u \in \dom^l_\infty(\Dir)$, let $\set{(U_i, \psi_i)}$ be a pre-compact open cover
of $\cM$ by locally comparable trivialisations/charts for both $\cM$ and $\cV$. 
Let $\set{\eta_i}$ be a smooth partition of unity subordinate to $\set{U_i}$. 
On writing $u_i = \eta_i u$,
it is clear that $u = \sum_{i=1}^\infty u_i$ 
pointwise almost-everywhere and also in $\Lp{2}$.

We show that $u_i \in \dom^l_\infty(\Dir)$.
For that, observe that by \ref{D:Reg}, we obtain $u \in \Sob[loc]{l,2}(\cV)$.
Fix a choice of smooth frame $\set{e_{i,j}}_{j=1}^N$ in $U_i$,
and the flat connection $\conn$ with respect to $\set{e_{i,j}}$ inside $U_i$.
Also, we have 
$$ \conn^l(\eta_i u) = \sum_{j=0}^l C^j_l \conn^j \eta_i \tensor \conn^{l - j} u,$$
where $C^j_l = \frac{l!}{j!(l - j)!}$. 
Since $\eta_i$ is smooth, $\conn^j \eta_i$ is bounded inside $U_i$ for all $j > 0$, 
by considering a sequence $u^n_i \to u_i$ in $\norm{\mdot}_{\Sob{1,2}(U,\cV)}$
with $u^n_i \in  \Ck{\infty}(U_i, \cV)$, we deduce that $u_i \in \Sob{1,2}(U,\cV)$.
Therefore, by Lemma \ref{Lem:Local}, we obtain that
$u_i \in \dom_0^l(\Dir) \subset \dom^l_\infty(\Dir)$.

Define $u_i^\epsilon = (J_\epsilon \convolve u_i^j)\ e_{i,j}$,
for $\epsilon > 0$, where $J_\epsilon$ is the standard symmetric mollifier. 
We can choose $\epsilon < \epsilon_i$ so that $\spt u_i^\epsilon \subset U_i$.
By Lemma \ref{Lem:Local}, there is a constant $C_i$ such that
$$ \norm{\Dir^l u_i - \Dir^l u_i^\epsilon} \
	\leq C_i \cbrac{\sum_{j=1}^l\norm{\conn^j (u_i - u_i^\epsilon)}_{\Lp{2}(U_i, \cV)} + \norm{u_i - u_i^\epsilon}_{\Lp{2}(U_i,\cV)}}.$$ 
Since we assume that $U_i$ satisfy the local comparability condition,
setting $f_i = u_i - u_i^\epsilon$, we have
$$\norm{\conn^j f_i}_{\Lp{2}(U_i, \cV)}^2 
	= \int_{U_i} \modulus{\conn^j f_i(x)}_{(\mg \tensor \mh)(x)}^2\ d\mu_\mg
	\lesssim \int_{U_i} \modulus{\conn^j f_i(x)}_{\delta(x)}^2\ d\pullb{\psi}_i\Leb(x),$$
where $\pullb{\psi}_i \Leb$ is the pullback of the Lebesgue measure inside $U_i$.
By standard mollification theory, we have that the right hand side of this expression 
tends to zero as $\epsilon \to 0$. Also, observe that $u_i^\epsilon \in \Sp_l(\Dir)$.

Fix $\delta > 0$ and for each $i$, choose $\delta_i > 0$ such that
$$ \norm{\Dir^l u_i - \Dir^l u_i^{\delta_i}} \leq \frac{\delta}{2^i}.$$
For each $x \in \cM$,  define
$u^\delta(x) = \sum_{i = 1}^\infty u_i^{\delta_i}(x)$,
 this is well-defined because $\set{\eta_i}$ is locally finite.
It is easy to see that $u^\delta \in \Ck{\infty}(\cV)$,
and to show that $u^\delta \in \Sp_l(\Dir)$, it suffices
to show that $\Dir^l u^\delta  \in \Lp{2}(\cV)$.
Fix $v \in \Ck[c]{\infty}(\cV)$ and note that
$$(\Dir^l u^\delta)(v) = \inprod{u^\delta, \Dir^l v} = \inprod{u^\delta - u, \Dir^l v}  + \inprod{u, \Dir^l v}.$$
By assumption, we have that $\modulus{\inprod{u, \Dir^l v}} \lesssim \norm{v}$,
and so we are reduced to showing that $\modulus{\inprod{u^\delta - u, \Dir^l v}} \lesssim \norm{v}.$
Since $\Dir$ is local, $\Dir^l$ is local, and 
we have that $\spt \Dir^l v \subset \spt v$. Choose $M \in \Na$
such that $\sum_{i = 1}^M \eta_i = 1$ on $\spt v$, we then have
\begin{multline*}
\modulus{\inprod{u^\delta - u, \Dir^l v}}
	= \modulus{\inprod{ \sum_{i=1}^N u_i^{\delta_i} - \sum_{i = 1}^N u_i, \Dir^l v}}
	\leq \sum_{i= 1}^N \modulus{\inprod{u_i^{\delta_i} - u_i, \Dir^l v}}\\
	\leq \sum_{i=1}^N \frac{\delta}{2^i} \norm{v} 
	\leq \delta \norm{v}.
\end{multline*}
Thus, $u^\delta \in \Sp_l(\Dir)$. By a similar argument,
we obtain that $u^\delta \to u$ in
the graph norm $\norm{\mdot}_{\Dir^l}  = \norm{\Dir^l \mdot}+ \norm{\mdot}$.
\end{proof}

With the aid of this proposition, we prove Theorem \ref{Thm:FirstMain}. 
\begin{proof}[Proof of Theorem \ref{Thm:FirstMain}]
If $\dom^l_0(\Dir) = \dom^l(\Dir)$, then $(\Dir^l)_D = (\Dir^l)_N = \adj{(\Dir_c^l)}$, 
where the last equality follows from Proposition \ref{Prop:DomMax}. That is,
$(\Dir^l)_D = \adj{(\Dir^l)}_D$, which shows that $(\Dir^l)_D$ is self-adjoint.

Now, if $(\Dir^l)_D$ is self-adjoint, we have that $(\Dir^l)_D = \adj{(\Dir^l)_D} = (\Dir^l_N)$
and so $(\Dir^l)_N$ is self-adjoint. 
If $(\Dir^l)_N$ is self-adjoint, since $\Sp_l(\Dir) \subset \dom(\Dir^l)_N$,
it follows that \eqref{lneg} is satisfied.

Recall the notation $\Dir_2^l = \Dir_l$ with domain $\dom(\Dir_2^l) = \Sp_l(\Dir)$.
Assuming the negligible boundary condition, we obtain that 
$(\Dir^l)_N \subset \adj{(\Dir_2^l)}$. But since $\Dir_c^l \subset \Dir_2^l$,
we have that $\adj{(\Dir_2^l)} \subset \adj{(\Dir_c^l)} = (\Dir^l_N)$
by Proposition \ref{Prop:DomMax}. So, $\adj{(\Dir_2^l)} = \adj{(\Dir_c^l)}$
and so by considering the bi-adjoint, we obtain that $(\Dir^l)_N = (\Dir^l)_D$.

The essential self-adjointness of  $\Dir_c^l$ is equivalent to the 
self-adjointness of $(\Dir^l)_D$.
\end{proof}
\section{Density properties from higher to lower powers}
\label{S:HtoL}
\label{S:core}

In this section, we prove that, if $(\Dir^m)_0 = \close{\Dir^m_c}$
is self-adjoint, then so is every power $(\Dir^l)_0$ for 
$l = 1, \dots, m$. The reasons for this involve no geometry, but only 
properties of the operators, and to highlight that, we will keep 
the presentation sufficiently abstract. The way we will proceed
is to move from the operator $(\Dir^m)_0$ to the operator
$\modulus{(\Dir^m)_0}$, then use results from interpolation 
theory to assert that $\modulus{(\Dir^m)_0^\alpha} = (\sqrt{ (\Dir^m)_0^2})^\alpha$ 
has $\Ck[c]{\infty}(\cV)$ as a core. We then employ an alternative
argument to show that this is a core for $\Dir_0$, then
use functional calculus and operator theory to assert
that $\modulus{(\Dir^m)_0^\alpha} = \modulus{\Dir_0}^{m \alpha}.$

The functional calculus we use here is the holomorphic
functional calculus. For self-adjoint operators $T$, this functional 
calculus is given by: 
$$ \psi(T)u = \frac{1}{2\pi} \oint_{\gamma} \psi(\zeta)(\zeta \iden -T)^{-1}u\ d\zeta,$$
where $\gamma$ is a curve cutting the spectrum at zero and infinity
inside a bisector. The functions $\psi$
are holomorphic on a bisector and 
decay sufficiently rapidly at $0$ and at $\infty$.
The functional calculus can be extended to 
bounded holomorphic functions $f$ on a bisector. 
A detailed exposition of these ideas can be found in
\cite{ADMc} by Albrecht, Duong, McIntosh and in
the book \cite{Haase} by Haase. 

\begin{lemma}
\label{Lem:Core1}
Let $T$ be a non-negative self-adjoint operator on 
a Hilbert space $\Hil$, and let $\sC \subset \dom(T)$
be a core for $T$. Then, $\sC$ is a core for 
$T^\alpha$ for all $\alpha \in [0,1]$. 
\end{lemma}
\begin{proof}
The case $\alpha  = 0,\ 1$ are easy, so we fix $\alpha \in (0,1)$.

First we show that  $\dom(T)$ is dense in $\dom(T^\alpha)$.
Theorem 6.6.1 in \cite{Haase} yields that the 
real-interpolation space, $(\Hil, \dom(T))_{\alpha, p}$ has $\dom(T)$
as a dense subspace for $\alpha \in (0,1)$ and $p \in [1, \infty)$.
By Theorem 4.3.12 in \cite{Lunardi}, on choosing $p = 2$, 
we obtain that $(\Hil, \dom(T))_{\alpha, 2} = \dom(T^\alpha)$.

Next, note that $\dom(T) \subset \dom(T^\alpha)$
by functional calculus
since $T^\alpha = (\iden + T) f_\alpha(T)$, where
$f_\alpha(\zeta) = \zeta^\alpha/(1 + \zeta)$. Moreover, 
since $f_\alpha(T) \in \bddlf(\Hil)$,
we get that $\norm{T^\alpha u} \lesssim \norm{(\iden + T)u}$
for all $u \in \dom(T)$.

Fix $u \in \dom(T^\alpha)$, since we have
already proved that $\dom(T)$ is dense in $\dom(T^\alpha)$,
we can find  $v_n \in \dom(T)$ such that
$$\norm{u - v_n} \leq \frac{1}{2^n}\quad \text{and}\quad 
	\norm{T^\alpha u - T^\alpha v_n} \leq \frac{1}{2^n}.$$
Moreover, since $\sC$ is a core for $T$, we can find
$u_n \in \sC$ such that
$$\norm{u_n - v_n} < \frac{1}{2^n}\quad \text{and}\quad \norm{Tv_n - Tu_n} < \frac{1}{2^n}.$$ 
Thus, on combining this with the estimate $\norm{T^\alpha u} \lesssim \norm{(\iden + T) u}$
for $u \in \dom(T)$, and since $v_n \in \dom(T)$,  
\begin{multline*}
\norm{T^\alpha u - T^\alpha u_n} 
	\leq \norm{T^\alpha u_n - T^\alpha v_n} + \norm{T^\alpha v_n - T^\alpha u} \\
	\lesssim \norm{(\iden + T)(u_n - v_n)} + \frac{1}{2^n}
	\leq \norm{u_n - v_n} + \norm{T u_n - T v_n} + \frac{1}{2^n}
	< \frac{3}{2^n}.
\end{multline*}
This shows that $\sC$ is a core for $T^\alpha$.
\end{proof}

In the following lemma, we note that for a self-adjoint
operator $T$, $\modulus{T} = \sqrt{T^2}$.

\begin{lemma}
\label{Lem:PowerMod}
Let $T$ be self-adjoint on $\Hil$. Then, for any integer $k \geq 0$,
$\dom(\modulus{T}^k) = \dom(T^k)$ with $\norm{\modulus{T}^k u} = \norm{T^ku}$.
\end{lemma}
\begin{proof}
Consider the functions 
$$f_1(\zeta) = \frac{\modulus{\zeta}^k}{\imath + \zeta^k},
\quad\text{and}\quad
f_2(\zeta) = \frac{\zeta^k}{1 + \modulus{\zeta}^k}.$$
Each such function is holomorphic, bounded, and hence, 
$f_1(T) = \modulus{T}^k(\imath \iden + T^k)^{-1} \in \bddlf(\Hil)$
and $f_2(T) = T^k(\iden + \modulus{T}^k)^{-1} \in \bddlf(\Hil)$. 
This shows that $\dom(T^k) = \dom(\modulus{T}^k)$ in graph norm.

To obtain the equivalence of norms, we note that by 
the self-adjointness of $T$, $\norm{Tu} = \norm{\modulus{T}u}$.
For higher powers, 
$$\norm{T^k u} = \norm{\modulus{T}T^{k-1} u} = \norm{T^{k-1}\modulus{T}u} = \dots = \norm{\modulus{T}^k u},$$
where the second equality follows from functional calculus.
\end{proof}

For the next lemma, we specialise to the 
operator in question.

\begin{lemma}
\label{Lem:HtoF}
Let $\Dir$ be an operator satisfying \ref{D:First}-\ref{D:Last}. 
If $(\Dir^l)_D$ is self-adjoint, then $\Dir_D$ is also self-adjoint.
\end{lemma}
\begin{proof}
First, we consider the case that $l = 2$. The symmetry condition
\ref{D:Close} implies that $\Dir_D$
exists and $\Dir \subset \Dir_D$ so that $\Dir^2 \subset (\Dir_D)^2$. 
Since $(\Dir_D)^2$ is closed, $(\Dir^2)_D \subset (\Dir_D)^2$ and hence,
by Lemma 3 in \cite{Bernau} (with $T = \Dir_D$ and $S = (\Dir_D)^2$), 
$\Dir_D$ is self-adjoint. Thus $\Dir_D^2 = (\Dir_D)^2$.

Now, for $l = 2^m$, we replace $\Dir$ by $\Dir^{2^{m-1}}$ and
by this argument, we obtain that $(\Dir^{2^{m-1}})_D$ is 
self-adjoint whenever $(\Dir^{2^m})_D$ is self-adjoint. 
Repeating this procedure $m$-times, we obtain that $\Dir_D$ 
is self-adjoint.

Next, suppose that $l$ is an odd number. Assume for
contradiction that $\Dir_D$ is not self-adjoint. 
 
Then $\nul(\adj{\Dir}_D + \imath) \neq 0$ or $\nul(\adj{\Dir}_D - \imath) \neq 0$. 
Without loss of generality, we assume $\nul(\adj{\Dir}_D - \imath) \neq 0$ and 
we can find a non-zero vector $v \in \dom(\adj{\Dir}_D)$ such
that $\adj{\Dir}_D v = \adj{\Dir} v = \imath v$. 

By Proposition \ref{Prop:DomMax}, we have that
$\dom(\adj{\Dir}) = \dom(\Dir_N) = \dom_\infty(\Dir)$,
so $\Dir_N v = \imath v$ which implies that $v \in \dom(\Dir_N^a)$ 
for all $a \geq 1$. Invoking \ref{D:Close}, 
since $v \in \dom(\Dir_N^a)$ for all $a \leq l$,  
$$
\inprod{\Dir_N^l v, u} = \inprod{\Dir_N^{l-1} v, \Dir_c u} 
	= \dots = \inprod{v, \Dir_c^l u}$$
for all $u \in \Ck[c]{\infty}(\cV)$. That is, 
$v \in \dom(\adj{\Dir^l_c}) = \dom_\infty^l(\Dir)$. 
But $(\Dir^l)_D$ is self-adjoint, 
$\dom_\infty^l(\Dir) = \dom((\Dir^l)_D)$, and so we have
$(\Dir^l)_D = \Dir_N^l v = -\imath v$, which contradicts
that $(\Dir^l)_D$ is self-adjoint.

For a general $l$, write $l = a \cdot 2^b$ for $a$ odd,
we obtain that $\Dir_D$ is self-adjoint whenever $(\Dir^l)_D$
is self-adjoint.
\end{proof}

With the aid of this, we obtain a proof of Theorem \ref{Thm:FirstMain2}.

\begin{proof}[Proof of Theorem \ref{Thm:FirstMain2}]
First, $\Dir \subset \Dir_D$ yields that $\Dir^l \subset \Dir_D^l$
and so $(\Dir^l)_D \subset \Dir_D^l$. The operator $(\Dir^l)_D$ is self-adjoint
by hypothesis and $\Dir_D^l$ is self-adjoint since $\Dir_D$ is self-adjoint.
Then, by  Lemma \ref{Lem:HtoF}, we have that $(\Dir^l)_D = \Dir^l_D$. 

By Lemma \ref{Lem:Core1}, $\Ck[c]{\infty}(\cV)$
is a core for $\modulus{(\Dir^l)_D}^\alpha = \modulus{\Dir^l_D}^\alpha$.
On considering the functions
$$ f_1(\zeta) = \frac{1}{1 + (\modulus{\zeta}^l)^\alpha}
\quad\text{and}\quad
f_2(\zeta) = \frac{1}{1 + \modulus{\zeta}^{l\alpha}},$$
and observing that $f_1 = f_2$ is bounded holomorphic on $\C$,
we obtain 
$$(1 + (\modulus{\Dir_D}^l)^\alpha)^{-1} = f_1(\Dir) 
	= f_2(\Dir) =  (1 + \modulus{\Dir_D}^{^l\alpha})^{-1},$$
and hence, $(\modulus{\Dir_D}^l)^\alpha = \modulus{\Dir_D}^{l\alpha}.$
Thus, we have that $\Ck[c]{\infty}(\cV)$ is a core
for $\modulus{\Dir_D}^{l \alpha}$ for $\alpha \in [0,1]$. 

On setting $\alpha = k/l$ for integers
$k=1, \dots, l$,  by Lemma \ref{Lem:PowerMod}, we obtain that
$\Ck[c]{\infty}(\cV)$ is a core for $\Dir_D^k$.
\end{proof}
 
\section{Essential self-adjointness of powers in the complete setting} 
\label{S:Ess}

Throughout this section, we assume  that $\mg$ is a rough
metric  
that induces a \emph{complete length} space. 
We will denote the distance metric of this length space by $d$. 
Furthermore,  we assume 
the bundle $\cV$ is equipped with a bundle rough metric $\mh$.

As a background assumption,  we assume \ref{D:First}-\ref{D:Last} for
the operator $\Dir$ (i.e., for the case $l = 1$).

\begin{proposition}\label{lip_sym}
The formula 
$\langle \Dir_N u, v\rangle = \langle u, \Dir v\rangle$ 
holds whenever $u \in \dom(\Dir_N)$ and $v \in \Ck[c]{0,1}(\cM)$.
\end{proposition}
\begin{proof}
We note by \ref{D:Close} that $\inprod{\Dir u, v} = \inprod{u,\Dir_D v}$ for 
all $u \in \Sp_1$ and $v \in \dom(\Dir_D)$. Thus, it suffices to show
that $\Ck[c]{0,1}(\cV) \subset \dom(\Dir_D)$. For that, fix $v \in \Ck[c]{0,1}(\cV)$.
On fixing a smooth partition of unity subordinate to locally 
comparable precompact charts, we can further assume that $\spt v \subset U$,
where $U$ is a precompact open chart. Inside there, fix the flat
connection $\conn^U$, and hence, by \ref{D:Reg}, we can deduce that that whenever $\Dir u \in \Lp{2}(\cV)$
with $\spt u \subset U$,
implies that $u \in \Sob{1,2}(U, \cV)$ with the estimate
$\norm{\Dir u}_U \lesssim \norm{\conn^U u}_U + \norm{u}_U.$
Now, let $v_n = (J^{\frac{1}{n}} \convolve v^i) e_i$, where
$\set{e_i}$ is the frame in $U$ for which $\conn^F$ is flat
and $J^{\frac{1}{n}}$ is the standard symmetric mollifier
$J^\epsilon$ with $\epsilon = 1/n$.
Hence, we have that $\spt v_n \subset U$ for
$n$ large enough, and moreover, $v_n \in \Ck[c]{\infty}(\cV)$.
In particular, $\Dir v_n \in \Lp{2}(\cV)$ and moreover, $v_m \to v$ in $\Lp{2}(\cV)$.
 Thus,
$$ \norm{\Dir_D v_n - \Dir_D v_m} \lesssim \norm{\conn^U v_n - \conn^U v_m}_U + \norm{v_m - v_n}$$
and the right hand side tends to zero as $m,\ n \to \infty$. By the closedness
of $\Dir_D$, we have that $v \in \dom(\Dir_D)$.
 It is easy to see that this now extends for $u \in \dom(\Dir_N)$. 
\end{proof} 

When we apply this proposition in later parts,
we will be taking $u \in \Sp_k$ and considering the section $\eta u$, where $\eta \in \Ck[c]{\infty}(\cM)$. 
In this case we have that $\Dir^k(\eta u)$ is  $\Ck[c]{m - k}(\cV)$ and hence this lies
in $\dom(\Dir_D) \subset \dom(\Dir_N)$. 
 Thus, we can apply this proposition to this case.   

\subsection{Some preliminary constructions and remarks}\label{prelims}

 In \cite{Chernoff}, Chernoff defines  the  \emph{local velocity of propagation} by:  
\begin{equation*}
c(x) = \sup\{\modulus{\Mul_f(x)}_{\mathrm{op}} : f \in \Ck{\infty}(\cM), \vert \conn(f) (x)\vert_{\mg} = 1\},
\end{equation*}
and the \emph{velocity of propagation inside a ball} by
$c(r) = \esssup\set{c(x) : x \in B_r}.$ 
It is easy to see that $c(r)$ is non-negative and monotonically
increasing.

Since $\Dir = A^i \partial_i + B$ locally with $A_i \neq 0$ for some $i$,
we have that $c(r) \neq 0$  for all $r > 0$. 
Also, since the operator locally takes the form $A^i \partial_i + B$,
it is easy to see that $\Mul_f(x) u(x) =  (\partial_i f)(x)A^i u(x)$.
By covering the ball $B_r$ by a finite number of locally 
comparable precompact charts, and using that $A^i \in \Lp[loc]{\infty}$,
 we obtain: there exists $C_{B_r} < \infty$ such that 
 $\modulus{\Mul_f(x)} \leq C_{B_r} \modulus{\conn f(x)}$
for almost-every $x \in B_r$.
Therefore, for each $r > 0$, the velocity of propagation 
satisfies $c(r) < \infty$. 

If we take a smooth function  $f$,  we then find that for any smooth section $u$
\begin{equation}\label{point_mult_bound}
\vert \Mul_f(u)(x)\vert_h \leq \vert \conn(f)(x)\vert_{\mg} c(x)\vert u(x)\vert_h
\end{equation}
 for almost-every $x \in \cM$. 
Letting $\vert\vert \cdot \vert\vert_{B_{r}}$ denote the $\Lp{2}$ norm on the ball $B_r$, we  obtain 
\begin{equation}\label{L^2_mult_bound}
\vert\vert \Mul_f(u)\vert\vert_{B_{r}} \leq (\esssup\{\vert \conn(f)(x)\vert_{\mg} : x \in B_r\})\ c(r)\vert\vert 
u\vert\vert_{B_r}.
\end{equation}

The following construction is based on  Wolf's construction 
in  \S5 of  \cite{Wolf},  and adapted to the case of
rough metrics.

Fix a point $x_0$ in $\cM$, and for any point $x \in \cM$, 
let $\rho(x) := d(x, x_0)$, where $d$ denotes the
distance function  associated to the  length structure induced on $\cM$ by $\mg$. 
The triangle inequality shows us that 
$\vert \rho(y) - \rho(x) \vert \leq d(y,x)$ for any $x, y \in \cM$. Thus $\rho$ is  Lipschitz
and hence differentiable almost-everywhere. 
Inside a locally comparable 
chart $(U,\psi)$,  we have $C_U \geq 1$ such that
$\modulus{\conn \rho(x)}_{\mg(x)} \leq C_U$
for almost-every $x \in U$. 

For $r > 0$ we let $B_r$ denote the open ball centred about the fixed point $x_0$. 
By the assumption that our rough metric induces a complete length space,
we are able to apply the metric space version of the Hopf-Rinow theorem 
(see Proposition 3.7 in \cite{BR}) and obtain that  metric balls  $B_r$ are 
precompact.  As a consequence, on a ball 
$B_r$, there exists $C_{B_r} \in [1, \infty)$ such that
$\modulus{\conn\rho(x)} \leq C_{B_r}$ for almost-every $x \in B_r$. 
Define 
$$\tilde{d}(r) = \max (1,\ \esssup\set{ \modulus{\conn\rho(x)}: x \in B_r}),$$
and note that $\tilde{d}(r)$ is increasing in $r$. Also, 
define $e(r) = \tilde{d}(r)c(r)$ which is again an increasing 
function in $r$. 

Choose a smooth function $a : \R \rightarrow [0,1]$ such that $a(-\infty, 1] = 1$, $a[2,\infty) = 0$, and
such that $a$ is non-zero on  the interval   $(1,2)$.
Denote $M = \max\vert a'(t)\vert$. 
 For each $r > 0$, define  
$b_r : \cM \rightarrow [0,1]$ by 
$b_r(y) = a(\frac{\rho(y)}{e(r)r})$. Then $b_r = 1$ on $B_{e(r)r}$ and 
$\spt(b_r) \subseteq \overline{B}_{2e(r)r}$. 
 Furthermore  $b_r$ is non-zero on $B_{2e(r)r}$, and
since $B_{2e(r)r}$ is precompact, it follows that $b_r$ is compactly supported. The Lipschitz property
of $\rho$ implies that $b_r$ is Lipschitz  and  hence almost-everywhere differentiable. At points of differentiability, and points
where $\vert\cdot\vert_{\mg}$ is defined, we have
\begin{equation}\label{grad_bound_b_r}
\Big\vert \conn(b_r)\Big\vert^2_{\mg} = \Bigg\vert \frac{1}{e(r)r} a'\Bigg(\frac{\rho}{e(r)r}\Bigg)\conn(\rho)\Bigg\vert^2_{\mg} 
\leq \frac{M^2}{c(r)^2r^2}.
\end{equation}
We remind the reader that $\vert\cdot\vert_{\mg}$ is defined almost-everywhere so the above gradient bound holds almost-everywhere on $\cM$.
In particular,  $\esssup\set{ \modulus{\conn(b_r)(x)}: x \in \cM}$  
is bounded above by $M (c(r)r)^{-1}$.

Using this function $b_r$ in equation \eqref{L^2_mult_bound} we find that
\begin{equation}\label{L^2_bound_b_r}
\vert\vert \Mul_{b_r}(u)\vert\vert_{B_{2e(r)r}} \leq \frac{M}{c(r)r}c(r)\vert\vert u\vert\vert_{B_{2e(r)r}} 
= \frac{M}{r}\vert\vert u\vert\vert_{B_{2e(r)r}}.
\end{equation}

We will also be making use of the powers $b_r^k$ for $k \geq 1$. In this case we note that 
because $\conn(b_r^k) = kb_r^{k-1}\conn(b_r)$, we get the following estimate:
\begin{equation}\label{L^2_bound_b_r^k}
\vert\vert \Mul_{b_r^k}(u)\vert\vert_{B_{2e(r)r}} \leq \frac{kM}{r}\vert\vert b_r^{k-1}u\vert\vert_{B_{2e(r)r}}.
\end{equation}


\subsection{Essential self-adjointness of \Dir}

In this subsection we prove that $\Dir$ is essentially self-adjoint on $\Ck[c]{\infty}(\cV)$. The approach we will be taking is to
establish that $\Dir$ satisfies the negligible boundary condition.

\begin{theorem}\label{D_essen_self}
Let $\cV$ be a vector bundle with a bundle rough metric $\mh$ over a manifold $\cM$ with a rough 
metric $\mg$ inducing a complete length space. 
If the operator $\Dir$ satisfies \ref{D:First}-\ref{D:Last},
then it satisfies the negligible boundary condition \eqref{lneg}, 
and $\Dir$ on $\Ck[c]{\infty}(\cV)$ is essentially self-adjoint. 
\end{theorem}
\begin{proof}
We fix two arbitrary sections $u, v \in \Sp_1$, what we need to prove is that
\begin{equation*}
\langle \Dir u, v \rangle = \langle u, \Dir v \rangle.
\end{equation*}
In order to do this, we consider the truncated section $b_rv$ for some $r > 0$. This section is 
 Lipschitz with  compact support. We then have that
\begin{equation*}
\langle \Dir u, b_rv \rangle = \langle u, \Dir(b_rv)\rangle = \langle u, \Mul_{b_r}v \rangle + \langle u, b_r\Dir v\rangle
\end{equation*}
where the first equality comes from Proposition \ref{lip_sym}, and the second from 
 the fact that  $\Mul_{b_r} = [\Dir, b_r \iden]$. 

From \eqref{point_mult_bound} and \eqref{L^2_bound_b_r} we see that $\Mul_{b_r}v \rightarrow 0$ as $r \rightarrow \infty$. 
Furthermore, $b_ru \rightarrow u$ in $\Lp{2}$ as 
$r \rightarrow \infty$ by the dominated convergence theorem, and
similarly for $b_r\Dir v$. This implies that if we let $r \rightarrow \infty$ in the above equality, we obtain
\begin{equation*}
\langle \Dir u, v \rangle = \langle u, \Dir v \rangle.
\end{equation*} 
Since $u, v$ were arbitrary sections it follows that such an equality holds for all $u, v \in \Sp_1$, which
is the negligible boundary condition \eqref{lneg}. The conclusion follows from 
Theorem \ref{Thm:FirstMain}. 
\end{proof}


\subsection{Essential self-adjointness of $\Dir^k$}
\label{S:Ess-k}

In this subsection we will show that the powers 
$\Dir^k$ on $\Ck[c]{\infty}(\cV)$, for $1 \leq k \leq m+1$, are essentially self-adjoint
when the operator $\Dir$ has $\Ck{m}$ coefficients.  The case $k = 1$ holds even 
for operators with $\Lp[loc]{\infty}$ coefficients and  
was obtained in Theorem \ref{D_essen_self}.
 We do not expect the essential self-adjointness to hold in 
general for $\Dir^k$ with $k > m+1$ since 
even for smooth, compactly supported sections $u$,
we can only make sense of $\Dir^k u$ distributionally. 

Our approach will be based on certain local estimates over a ball of radius $r$. We will then
show that these local estimates prolong to global estimates on the whole manifold by taking the radius $r$
to infinity. These global estimates  allow us to prove the  negligible boundary condition, and from 
our previous work, establish essential self-adjointness. This idea of using local estimates
 to  establish  essential self-adjointness of an operator is inspired by Wolf's work on the Dirac operator in 
\cite{Wolf}.
More specifically, in Proposition 6.2 in \cite{Wolf}, Wolf establishes a certain global estimate of 
$\vert\vert \Dir(u)\vert\vert$ in terms of $\vert\vert \Dir^2(u)\vert\vert$ and $\vert\vert u\vert\vert$ and this, along
with some other facts, allows him to conclude the essential self-adjointness of $\Dir^2$, where $\Dir$ is the Dirac 
operator. We will show that when $u \in \Sp_{k+1}$ one can bound $\vert\vert \Dir^k(u)\vert\vert$ in terms
of  $\vert\vert \Dir^{k+1}(u)\vert\vert$ and $\vert\vert u\vert\vert$.  This implies that 
$\Sp_{k+1} \subseteq \Sp_k$, which
we will then exploit to prove the negligible boundary condition \eqref{lneg} for  $\Dir^{k+1}$.  
We should also mention that Proposition 6.2 in \cite{Wolf} is based on the method of 
Andreotti-Vesentini   (see \S6 in \cite{andreotti}). 

In addition to the assumptions of Theorem \ref{D_essen_self}, 
for the remainder of this section, we assume that $\Dir$
is a $\Ck{m}$ coefficient operator satisfying \ref{D:First}-\ref{D:Last} for $\Dir^k$
 with  $2 \leq k \leq m+1$.

\begin{proposition}\label{Dir_est}
For a section $u \in \Sp_2$, we have that for any $t_1, t_2 > 0$ that
\begin{equation}\label{Dir_glob_est}
\vert\vert \Dir(u)\vert\vert^2 \leq \frac{1}{2t_2}\vert\vert \Dir^2(u)\vert\vert^2 + 
\Big(\frac{t_1}{2} + \frac{t_2}{2}\Big)\vert\vert u\vert\vert^2.
\end{equation}
In particular, $\Dir(u) \in \Lp{2}(\cV)$ which implies $u \in \Sp_1$.
\end{proposition}
\begin{proof}
We will be using the  functions $b_r$ as constructed in \S\ref{prelims}. 
 We start  by proving the following:

If $t_1, t_2 > 0$ are given then there exists $r_1 = r_1(t_1)$ such that for all $r \geq r_1$
\begin{equation}\label{D_loc_est}
\Big(1 - \frac{2M^2}{t_1r^2}\Big)\vert\vert b_r\Dir(u)\vert\vert^2_{B_{2e(r)r}} 
\leq \frac{1}{2t_2}\vert\vert b_r^2\Dir^2(u)\vert\vert^2_{B_{2e(r)r}}  + 
\Big(\frac{t_1}{2} + \frac{t_2}{2}\Big)\vert\vert u\vert\vert^2_{B_{2e(r)r}} 
\end{equation}
and $\Big(1 - \frac{2M^2}{t_1r^2}\Big) > 0$.

 Fix a smooth section $u$ of $\cV$. Consider the compactly supported section $b_r^2u$. 
Since $b_r$ is
differentiable almost-everywhere, at points of differentiability, we have the following formula:
\begin{equation*}
\Dir(b_r^2u) = \Mul_{b_r^2}(u) + b_r^2\Dir(u).
\end{equation*}
We then fix a smooth compactly supported function $\eta$ such that $\eta = 1$ on $B_{3e(r)r}$ and $\eta = 0$ outside of 
$B_{4e(r)r}$. We have
\begin{equation*}
\vert\vert b_r\Dir(\eta u)\vert\vert^2 = \langle b_r\Dir(\eta u), b_r\Dir(\eta u) \rangle 
= \langle b_r^2\Dir(\eta u), \Dir(\eta u) \rangle = \langle \Dir(b_r^2\Dir(\eta u)), \eta u \rangle 
\end{equation*}
where we have used Proposition \ref{lip_sym} to get the last equality, noting that $b_r^2\Dir(\eta u) \in \Ck[c]{0,1}(M)$.
As $b_r^2$ has support contained in $\overline{B}_{2e(r)r}$, we have that the above inner products are zero on the
complement of $\overline{B}_{2e(r)r}$. This implies, using the fact that $\eta = 1$ on $B_{3e(r)r}$, that
\begin{equation*}
\vert\vert b_r\Dir(u)\vert\vert^2_{B_{2e(r)r + \epsilon}} = \langle \Dir(b_r^2\Dir(u)), u \rangle_{B_{2e(r)r + \epsilon}}
\end{equation*}
for every $\epsilon > 0$ sufficiently small. Therefore it must be true on the open ball $B_{2e(r)r}$, so we find
\begin{equation}\label{est_b_rD_1}
\begin{aligned}
\vert\vert b_r\Dir(u)\vert\vert^2_{B_{2e(r)r}} &= \langle \Dir(b_r^2\Dir(u)), u \rangle_{B_{2e(r)r}} 
\\&= \langle \Mul_{b_r^2}(\Dir u), u \rangle_{B_{2e(r)r}} + \langle b_r^2\Dir^2(u), u \rangle_{B_{2e(r)r}}.
\end{aligned}
\end{equation}
Using Cauchy-Schwarz and the estimate \eqref{L^2_bound_b_r^k}, for any $t_1 > 0$ we obtain the bound
\begin{equation}\label{est_b_rD_2}
\vert\langle \Mul_{b_r^2}(\Dir u), u \rangle_{B_{2e(r)r}}\vert \leq  
\frac{2M^2}{t_1r^2}\vert\vert b_r\Dir(u)\vert\vert^2_{B_{2e(r)r}} 
+ \frac{t_1}{2}\vert\vert u\vert\vert^2_{B_{2e(r)r}}.
\end{equation}
For any $t_2 > 0$, the Cauchy-Schwarz inequality also implies
\begin{equation}\label{est_b_rD_3}
\vert \langle b_r^2\Dir^2(u), u \rangle_{B_{2r}}\vert \leq 
\frac{1}{2t_2}\vert\vert b_r^2\Dir^2(u)\vert\vert^2_{B_{2r}} + \frac{t_2}{2}\vert\vert u\vert\vert^2_{B_{2r}}
\end{equation}
Using \eqref{est_b_rD_2} and \eqref{est_b_rD_3} in \eqref{est_b_rD_1} we obtain
\begin{equation*}
\vert\vert b_r\Dir(u)\vert\vert^2_{B_{2e(r)r}} \leq \frac{2M^2}{t_1r^2}\vert\vert b_r\Dir(u)\vert\vert^2_{B_{2e(r)r}} + 
\frac{1}{2t_2}\vert\vert b_r^2\Dir^2(u)\vert\vert^2_{B_{2e(r)r}} +
\Big(\frac{t_1}{2} + \frac{t_2}{2}\Big)\vert\vert u\vert\vert^2_{B_{2e(r)r}},
\end{equation*} 
which we can simplify to 
\begin{equation*}
\Big(1 - \frac{2M^2}{t_1r^2}\Big)\vert\vert b_r\Dir(u)\vert\vert^2_{B_{2e(r)r}} \leq 
\frac{1}{2t_2}\vert\vert b_r^2\Dir^2(u)\vert\vert^2_{B_{2e(r)r}} +
\Big(\frac{t_1}{2} + \frac{t_2}{2}\Big)\vert\vert u\vert\vert^2_{B_{2e(r)r}}.
\end{equation*} 
By choosing $r$ large enough we can find a $r_1 = r_1(t_1)$ such that $\Big(1 - \frac{2M^2}{t_1r^2}\Big) > 0$
for any $r \geq r_1(t_1)$, and hence we have proved \eqref{D_loc_est}.

Using the fact that $b_r \leq 1$ and $b_r = 1$ on $B_{r}$ we obtain
\begin{equation*}
\Big(1 - \frac{2M^2}{t_1r^2}\Big)\vert\vert \Dir(u)\vert\vert^2_{B_r} \leq 
\frac{1}{2t_2}\vert\vert \Dir^2(u)\vert\vert^2_{B_{2r}} +
\Big(\frac{t_1}{2} + \frac{t_2}{2}\Big)\vert\vert u\vert\vert^2_{B_{2r}}
\end{equation*}
for any $r \geq r_1$. Letting $r \rightarrow \infty$ then establishes \eqref{Dir_glob_est} by the monotone convergence
theorem.
\end{proof}

Proposition \ref{Dir_est} immediately allows us to prove that $\Dir^2$ must be essentially self-adjoint.

\begin{proposition}\label{Dir^2_essen_self}
$\Dir^2$ is essentially self-adjoint on $\cM$.
\end{proposition}
\begin{proof}
The above Proposition \ref{Dir_est} implies that $\Sp_2 \subseteq \Sp_1$. Furthermore, Theorem \ref{D_essen_self}
implies that $\Dir$ satisfies the negligible boundary condition \eqref{lneg} on $\cM$. Therefore, if we take $u, v \in \Sp_2$ we have
\begin{equation*}
\langle \Dir^2u, v\rangle = \langle \Dir u, \Dir v\rangle = \langle u, \Dir^2v\rangle,
\end{equation*}
which is the negligible boundary condition \eqref{lneg} for $\Dir^2$.  From  Theorem \ref{Thm:FirstMain} 
it follows that $\Dir^2$ is essentially
self-adjoint on $\cM$.  
\end{proof}

\begin{remark}\label{rem_Dir_est}
We observe that part of the hypothesis of Proposition \ref{Dir_est} involved assuming that our section
$u \in \Sp_2$. This was done so that in the \emph{global} estimate \eqref{Dir_glob_est}, we knew that the right
hand side was finite, hence this gives  us that
 the left hand side is finite. 
In the case of the \emph{local} estimate
\eqref{D_loc_est}, the assumption that $u \in \Sp_2$ is not needed.
It suffices to assume that $u$ is smooth (in fact $\Ck{2}$ is enough). 
This is because we can  
fix a smooth compactly supported function $\eta$ taking
the value $\eta = 1$ on $B_{3e(r)r}$ and which  vanishes outside of 
$B_{4e(r)r}$. Then,  $\eta u \in \Sp_2$ and we can use the fact 
that $\eta u = u$ on $B_{2e(r)r}$ where the local estimate holds. 
We will see that this is a crucial observation 
in the argument to obtain local estimates for higher $\Dir^ku$.
\end{remark}

 In Proposition  \ref{Dir_est}, in order to obtain our required global estimate 
\eqref{Dir_glob_est}, 
we first proved an estimate on the ball $B_{2e(r)r}$ before taking a limit 
$r \rightarrow \infty$.  In order to prove the essential self-adjointness of higher
powers $\Dir^{k+1}$ for $k \geq 2$, we will proceed along the same lines. Our goal will therefore be to obtain a
similar bound on $\vert\vert \Dir^{k}(u)\vert\vert^2$ in terms of 
$\vert\vert \Dir^{k+1}(u)\vert\vert^2$ and $\vert\vert u\vert\vert^2$, 
over  balls of sufficiently large radius. 
We will give the details for the case $k = 2$, which involves some slight modifications of the above $k = 1$ case,
and then we will show how to do the general case via induction.

We start by proving the following local estimate.

\begin{proposition}\label{Dir^2_loc_prop}
For $u \in \Sp_3$, given any $t_1, t_3, t_4 > 0$ we can choose $r$ and $t_2$ sufficiently large so that 
\begin{enumerate}[(i)] 
\item $C_1(r, t_1) = \Big(1 - \frac{2M^2}{t_1r^2}\Big) > 0$, and

\item $C_2(r, t_1, t_2, t_3, t_4) = 
\Big(1 - \frac{8M^2}{t_3r^2} - \Big(\frac{t_3}{2} + \frac{t_4}{2}\Big)C_1(r, t_1)^{-1}\frac{1}{2t_2}\Big) > 0.$
\end{enumerate}
Moreover, the following estimate holds: 
\begin{equation}\label{Dir^2_loc_est}
\begin{aligned}
C_2\vert\vert b_r^2\Dir^2(u)\vert\vert^2_{B_{2e(r)r}} &\leq 
\frac{1}{2t_4}\vert\vert b_r^3\Dir^3(u)\vert\vert^2_{B_{2e(r)r}}  \\
&\qquad\qquad+
\Big(\frac{t_1}{2} + \frac{t_2}{2}\Big)\Big(\frac{t_3}{2} + \frac{t_4}{2}\Big)C_1^{-1}\vert\vert 
u\vert\vert^2_{B_{2e(r)r}}.
\end{aligned}
\end{equation}
\end{proposition} 
\begin{proof}
We start by estimating $\vert\vert b_r^2\Dir^2(u)\vert\vert^2_{B_{2e(r)r}}$. 
As in the proof
of Proposition \ref{Dir_est} we fix a  smooth 
compactly supported function $\eta$ such that $\eta = 1$ on $B_{3e(r)r}$ and $\eta = 0$ outside of 
$B_{4e(r)r}$. We then have
\begin{align*}
\langle b_r^2\Dir^2(\eta u), b_r^2\Dir^2(\eta u) \rangle
&= \langle \Dir(b_r^4\Dir^2(\eta u), \Dir(\eta u)\rangle \\
&= \langle \Mul_{b_r^4}(\Dir^2(\eta u)), \Dir(\eta u)\rangle + \langle b_r^4\Dir^3(\eta u), \Dir(\eta u)\rangle \\
&= \langle \Mul_{b_r^4}(\Dir^2(\eta u)), \Dir(\eta u)\rangle + \langle b_r^3\Dir^3(\eta u), b_r\Dir(\eta u)\rangle
\end{align*}
where we have used Proposition \ref{lip_sym} to obtain the first equality 
 (noting that $\Dir(\eta u) \in \Ck[c]{m-1}(E)$, hence is in $\dom(\Dir_D)$), 
and the fact that at points of differentiability of $b_r^4$,
$\Dir(b_r^4\Dir^2(\eta u)) = \Mul_{b_r^4}(\Dir^2(\eta u) + b_r^4\Dir^3(\eta u)$. Since $b_r$ has support contained
$\overline{B}_{2e(r)r}$ and $\eta = 1$ on $B_{3e(r)r}$, the above gives
\begin{equation*}
\langle b_r^2\Dir^2(u), b_r^2\Dir^2(u) \rangle_{B_{2e(r)r + \epsilon}} =
\langle \Mul_{b_r^4}(\Dir^2(u)), \Dir(u)\rangle + \langle b_r^3\Dir^3(u), b_r\Dir(u)\rangle_{B_{2e(r)r + \epsilon}}
\end{equation*}
for any $\epsilon > 0$ sufficiently small. By taking $\epsilon \rightarrow 0$ we obtain
\begin{equation*}
\langle b_r^2\Dir^2(u), b_r^2\Dir^2(u) \rangle_{B_{2e(r)r}} =
\langle \Mul_{b_r^4}(\Dir^2(u)), \Dir(u)\rangle{B_{2e(r)r}} + \langle b_r^3\Dir^3(u), b_r\Dir(u)\rangle_{B_{2e(r)r}}.
\end{equation*}

Applying Cauchy-Schwarz, the estimate \eqref{L^2_bound_b_r^k}, 
and the fact that $b_r$ has no zeroes in $B_{2e(r)r}$, we obtain
\begin{align*}
\big\vert \langle \Mul_{b_r^4}(\Dir^2(u)), \Dir(u)\rangle_{B_{2e(r)r}}\big\vert &=
\Big\vert \Big\langle \frac{1}{b_r}\Mul_{b_r^4}(\Dir^2(u)), b_r\Dir(u)\Big\rangle_{B_{2e(r)r}}\Big\vert \\
&\leq \frac{8M^2}{t_3r^2}\Big\vert\Big\vert \frac{b_r^3}{b_r}\Dir^2(u)\Big\vert\Big\vert^2_{B_{2e(r)r}} + 
\frac{t_3}{2}\vert\vert b_r\Dir(u)\vert\vert^2_{B_{2e(r)r}} \\
&= \frac{8M^2}{t_3r^2}\vert\vert b_r^2\Dir^2(u)\vert\vert^2_{B_{2e(r)r}} + 
\frac{t_3}{2}\vert\vert b_r\Dir(u)\vert\vert^2_{B_{2e(r)r}}.
\end{align*}
Similarly we get the bound  
\begin{equation*}
\big\vert \langle b_r^3\Dir^3(u), b_r\Dir(u)\rangle_{B_{2e(r)r}}\big\vert \leq 
\frac{1}{2t_4}\vert\vert b_r^3\Dir^3(u)\vert\vert^2_{B_{2e(r)r}}
+ \frac{t_4}{2}\vert\vert b_r\Dir(u)\vert\vert^2_{B_{2e(r)r}}.
\end{equation*} 
Using these two estimates we get
\begin{multline*}
\vert\vert b_r^2\Dir^2(u)\vert\vert^2_{B_{2e(r)r}} 
\leq \frac{8M^2}{t_3r^2}\vert\vert b_r^2\Dir^2(u)\vert\vert^2_{B_{2e(r)r}} +
\Big(\frac{t_3}{2} + \frac{t_4}{2}\Big)\vert\vert b_r\Dir(u)\vert\vert^2_{B_{2e(r)r}} \\
+ \frac{1}{2t_4}\vert\vert b_r^3\Dir^3(u)\vert\vert^2_{B_{2e(r)r}}.
\end{multline*}
 
We now want to plug in our local estimate for the $\vert\vert b_r\Dir(u)\vert\vert^2_{B_{2e(r)r}}$ term
occurring on the right. We remind the reader that at first sight  this seems impossible. 
In obtaining a local estimate for $\Dir(u)$ we assumed $u \in \Sp_2$, and our hypothesis at this point is that
$u \in \Sp_3$.  However, we do not know, a priori, that $u \in \Sp_2$ and in fact, this is what we will prove when we obtain the 
associated \emph{global} estimate.  However, as we observed in Remark \ref{rem_Dir_est}, we actually do
not  require
$u$ to be in $\Sp_2$ for the \emph{local} estimate.  It suffices
to ask that $u$ be smooth. 
 Thus, on substituting the  local estimate \eqref{D_loc_est} into the above inequality, we obtain: 
\begin{align*}
\vert\vert b_r^2\Dir^2(u)\vert\vert^2_{B_{2e(r)r}} 
\leq \frac{8M^2}{t_3r^2}&\vert\vert b_r^2\Dir^2(u)\vert\vert^2_{B_{2e(r)r}}  \\
&+
\Big(\frac{t_3}{2} + \frac{t_4}{2}\Big)\Big(1 - \frac{2M^2}{t_1r^2}\Big)^{-1}\frac{1}{2t_2}\vert\vert 
b_r^2\Dir^2(u)\vert\vert^2_{B_{2e(r)r}} \\ 
&+ \Big(\frac{t_1}{2} + \frac{t_2}{2}\Big)\Big(\frac{t_3}{2} + \frac{t_4}{2}\Big)
\Big(1 - \frac{2M^2}{t_1r^2}\Big)^{-1}\vert\vert u\vert\vert^2_{B_{2e(r)r}}  \\
&+ \frac{1}{2t_4}\vert\vert b_r^3\Dir^3(u)\vert\vert^2_{B_{2e(r)r}}.
\end{align*}
We can re-write this to give
\begin{align*}
\Bigg(1 - \frac{8M^2}{t_3r^2} &- \Big(\frac{t_3}{2} + \frac{t_4}{2}\Big)C_1^{-1}\frac{1}{2t_2}\Bigg)\vert\vert 
 b_r^2\Dir^2(u)\vert\vert^2_{B_{2e(r)r}}  \\ &\leq   
\frac{1}{2t_4}\vert\vert b_r^3\Dir^3(u)\vert\vert^2_{B_{2e(r)r}} +
\Big(\frac{t_1}{2} + \frac{t_2}{2}\Big)\Big(\frac{t_3}{2} + \frac{t_4}{2}\Big)
C_1^{-1}\vert\vert u\vert\vert^2_{B_{2e(r)r}},
\end{align*}
where we remind the reader that $C_1(r, t_1) = \Big(1 - \frac{2M^2}{t_1r^2}\Big)$.

We are now left with showing that we can choose $r$ and $t_2$ sufficiently large so that the coefficients in the
estimate are positive. It is easy to see that for large $r$ we have $C_1 > 0$. We take 
$t_2 = \lambda_2(t_3 + t_4)$, where $\lambda_2$ is a parameter,  and observe  that by taking $r$ 
and $\lambda_2$ sufficiently large, we can make the coefficient $C_2$ on the left hand side positive.
This finishes the proof.
\end{proof}

Observe that if we take $r \rightarrow \infty$, 
we have that $C_1 \rightarrow 1$, and 
$C_2 \rightarrow 1 - \Big(\frac{t_3}{2} + \frac{t_4}{2}\Big)\frac{1}{2t_2}$. So by choosing $t_2$ appropriately,
 $1 - \cbrac{\frac{t_3}{2} + \frac{t4}{2}}\frac{1}{2t_2}$  
can always be made greater than zero. Therefore by taking $r \rightarrow \infty$ in the above local estimate, we
obtain the following global estimate. 
\begin{theorem}
For $u \in \Sp_3$, given any $t_1, t_3, t_4 > 0$ we can choose $t_2$ sufficiently large so that 
$1 - \Big(\frac{t_3}{2} + \frac{t_4}{2}\Big)\frac{1}{2t_2}  > 0$, and 
\begin{equation}\label{Dir^2_est}
\Bigg(1 - \Big(\frac{t_3}{2} + \frac{t_4}{2}\Big)\frac{1}{2t_2} \Bigg)
\vert\vert \Dir^2(u)\vert\vert^2 \leq 
\frac{1}{2t_4}\vert\vert \Dir^3(u)\vert\vert^2 +
\Big(\frac{t_1}{2} + \frac{t_2}{2}\Big)\Big(\frac{t_3}{2} + \frac{t_4}{2}\Big)\vert\vert 
u\vert\vert^2.
\end{equation}
In particular, $\vert\vert \Dir^2(u)\vert\vert < \infty$, and hence $\Sp_3 \subseteq \Sp_2$.
\end{theorem}

The above theorem can now be used to establish essential self-adjointness of $\Dir^3(u)$.
\begin{proposition}\label{Dir^3_essen_self}
$\Dir^3$ is essentially self-adjoint.
\end{proposition}
\begin{proof}
The above theorem shows us that $\Sp_3 \subseteq \Sp_2$, and we know from Proposition \ref{Dir^2_essen_self} 
that $\Sp_2 \subseteq \Sp_1$. Therefore, if we take $u, v \in \Sp_3$ then
\begin{equation*}
\langle \Dir^3u, v\rangle = \langle \Dir u, \Dir^2v\rangle = \langle u, \Dir^3v \rangle 
\end{equation*}
which is the negligible boundary condition \eqref{lneg} for $\Dir^3$.  It follows from Theorem \ref{Thm:FirstMain} 
that $\Dir^3$ is essentially self-adjoint. 
\end{proof}

The general case of proving essential self-adjointness of $\Dir^{k+1}$ proceeds 
along  similar  lines.
In obtaining a local estimate for $\Dir^k(u)$, we will meet constants of the form
$C_k(r, t_1, t_2, \ldots, t_{2k-1}, t_{2k})$ that will depend on  a given  set of  positive numbers
$t_1, t_3, \ldots , t_{2k-1}, t_{2k}$, and constants $C_{k-1}(r, t_1, \ldots , t_{2(k-1)})$ that arise
from the local estimate for $\Dir^{k-1}(u)$. The goal will be to show that we can pick these constants to be
positive,  by making a suitable choice for $t_2, t_4, \ldots , t_{2(k-1)} > 0$.  
For example, in the case of obtaining a local estimate
for $\Dir^3(u)$ in terms of $\Dir^4(u)$ and $u$, the constant $C_3(r, t_1, t_2, t_3, t_4, t_5, t_6)$ takes the form
\begin{equation*}
C_3(r, t_1, t_2, t_3, t_4, t_5, t_6) = \Bigg(1 - \frac{18M^2}{t_5r^2} - \Big(\frac{t_5}{2} + \frac{t_6}
{2}\Big)C_2^{-1} \frac{1}{2t_4}\Bigg),
\end{equation*}
where we remind the reader that 
$$C_2(r, t_1, t_2, t_3, t_4) = 
\Big(1 - \frac{8M^2}{t_3r^2} - \Big(\frac{t_3}{2} + \frac{t_4}{2}\Big)C_1(r, t_1)^{-1}\frac{1}{2t_2}\Big),$$
 and $C_1(r, t_1) = \Big(1 - \frac{2M^2}{t_1r^2}\Big)$. The type of question we will be faced  with is: given
$t_1, t_3, t_5, t_6$ can we choose $r, t_2, t_4$ so that $C_i > 0$ for $1 \leq i \leq 3$? 
The idea would be to choose $r$ and $t_4$ sufficiently large so that $C_3 > 0$. The problem is that
as we vary $t_4$ the constant $C_2$ also changes, so there is  a non triviality that 
we need to prove.  The way in which we  proceed is to generalise the argument demonstrating 
how  $C_2$ can be made positive.  We define
$t_4 = \lambda_4(t_5 + t_6)$, where $\lambda_4$ is  a parameter to be chosen later.  We then define
$t_2 = \lambda_2(t_3 + t_4) = \lambda_2(t_3 + \lambda_4(t_5 + t_6))$, and substitute these into the formulas
for $C_2$ and $C_3$ to obtain
\begin{equation*}
C_2 = \Big(1 - \frac{8M^2}{t_3r^2} - C_1^{-1}\frac{1}{4\lambda_2}\Big),\quad
C_3 =  \Bigg(1 - \frac{18M^2}{t_5r^2} - C_2^{-1}\frac{1}{2\lambda_4}\Bigg). 
\end{equation*}
In the above formula for $C_3$, we can see that the $C_2^{-1}$ term does not depend on $\lambda_4$. 
Therefore, we  start
by choosing $r$ sufficiently large so that $C_1 > 0$.  Then we   choose $\lambda_2$ and $r$ sufficiently large
so that $C_2 > 0$, and finally choose $\lambda_4$ and $r$ sufficiently large so that $C_3 > 0$.

The general case of proving that the constants that come out can always be made positive follows in a similar
fashion. We will now prove a lemma that shows how to do this.

\begin{lemma}\label{constants}
Let $C_1(r, t_1) = 1 - \frac{2M^2}{t_1r^2}$, and for $i \geq 2$ recursively define the functions
\begin{equation*}
C_i(r, t_1, \ldots , t_{2i-1}, t_{2i}) 
= 1 - \frac{2i^2M^2}{t_{2i-1}r^2} - \Big(\frac{t_{2i-1}}{2} + \frac{t_{2i}}{2}\Big)
C_{i-1}(r, t_1, \ldots , t_{2(i-1)})^{-1}\frac{1}{2t_{2(i-1)}}. 
\end{equation*}
Then the functions $C_i$ for $i \geq 2$ satisfy the following two conditions.
\begin{itemize}
\item[(i)] $C_i(r, t_1, \lambda_2(t_3 + t_4), \ldots , \lambda_{2i-2}(t_{2i-1} + t_{2i}), t_{2i-1}, t_{2i})$
(where we insert $\lambda_{2j}(t_{2j+1} + t_{2j+2})$ in the 2j position $1 \leq j < i$) is independent of
$t_4, t_6, \ldots t_{2i-2}, t_{2i}$, and

\item[(ii)] $C_i(r, t_1, \lambda_2(t_3 + t_4), \ldots , \lambda_{2i-2}(t_{2i-1} + t_{2i}), t_{2i-1}, t_{2i}) > 0$ 
for $\lambda_2, \lambda_4, \ldots ,\lambda_{2i-2}, r$ sufficiently large.
\end{itemize}
In particular, given any positive $t_1, t_3, \ldots , t_{2i-1}, t_{2i}$ we can make
$C_i(r, t_1, \ldots , t_{2i-1}, t_{2i}) > 0$ for $t_2, t_4, \ldots , t_{2i-2}$ and $r$ sufficiently large.
\end{lemma}
\begin{proof}
We will prove this by induction. We have already seen that it is true for the $i = 2$ case (and in fact 
the $i = 3$ case). So assume it is true for all $2 \leq i \leq k-1$, we will prove it is true for $i = k$.

The formula for $C_k$ is given by
\begin{equation*}
C_k(r, t_1, \ldots , t_{2k-1}, t_{k}) 
= 1 - \frac{2k^2M^2}{t_{2k-1}r^2} - \Big(\frac{t_{2k-1}}{2} + \frac{t_{2k}}{2}\Big)
C_{k-1}(r, t_1, \ldots , t_{2(k-1)})^{-1}\frac{1}{2t_{2(k-1)}}. 
\end{equation*}
We observe that if we insert $\lambda_{2j}(t_{2j+1} + t_{2j+2})$ into the $2j$-th position for $1 \leq j < k-1$, we 
only affect the $C_{k-1}^{-1}$ in the above formula for $C_k$. Our induction hypothesis tells us
that this term can be made independent of $t_4, \ldots , t_{2(k-1)}$. If we then substitute 
$\lambda_{2(k-1)}(t_{2k-1} + t_{2k})$ into the $2(k-1)$-th position, 
 we see that we do not  affect the
$C_{k-1}^{-1}$ term, because this has been made independent of $t_{2(k-1)}$, and then we see that
we get rid of the $t_{2k}$ dependence arising in the other terms. In particular, it follows that
$C_k(r, t_1, \lambda_2(t_3 + t_4), \ldots , \lambda_{2k-2}(t_{2k-1} + t_{2k}), t_{2k-1}, t_{2k})$ is independent 
of $t_4, t_6, \ldots , t_{2(k-1)}, t_{2k}$, and this establishes that $C_k$ satisfies the first condition.

For the second condition, we observe that 
$C_k(r, t_1, \lambda_2(t_3 + t_4), \ldots , \lambda_{2k-2}(t_{2k-1} + t_{2k}), t_{2k-1}, t_{2k})$ is given by
\begin{multline*}
1 - \frac{2k^2M^2}{t_{2k-1}r^2} -
C_{k-1}(r, t_1, \lambda_2(t_3 + t_4), \ldots ,\\ \lambda_{2(k-2)}(t_{2(k-1)-1} + t_{2(k-1)}), 
t_{2(k-1)-1}, \lambda_{2(k-1)}(t_{2k-1} + t_{2k}))^{-1}\frac{1}{2\lambda_{2(k-1)}}.
\end{multline*}
Our induction hypothesis allows us to choose $\lambda_2, \lambda_4, \ldots , \lambda_{2(k-2)}$, and $r$ sufficiently 
large so that $C_i > 0$ for $2 \leq i \leq k-1$. Furthermore, the first condition of the induction hypothesis 
implies  that the constant
$$C_{k-1}(r, t_1, \lambda_2(t_3 + t_4), \ldots , 
	\lambda_{2(k-2)}(t_{2(k-1)-1} + t_{2(k-1)}), t_{2(k-1)-1}, t_{2(k-1)})$$ 
 does not depend on $t_{2(k-1)}$, which in turn allows us to conclude that
$$C_{k-1}(r, t_1, \lambda_2(t_3 + t_4), \ldots , \lambda_{2(k-2)}(t_{2(k-1)-1} + t_{2(k-1)}), 
t_{2(k-1)-1}, \lambda_{2(k-1)}(t_{2k-1} + t_{2k}))$$
 does not depend on $\lambda_{2(k-1)}$. 
Therefore, by taking $\lambda_{2(k-1)}$ and $r$ sufficiently large we can make it so that
\begin{equation*}
C_k(r, t_1, \lambda_2(t_3 + t_4), \ldots , \lambda_{2k-2}(t_{2k-1} + t_{2k}), t_{2k-1}, t_{2k}) > 0.
\end{equation*}
This establishes the second condition for $C_k$. It follows by induction that it is true for all
$i \geq 2$.

As for the last statement, simply observe that if we are given positive numbers
$t_1, t_3, \ldots , t_{2i-1}, t_{2i}$ we can substitute $t_{2j} = \lambda_{2j}(t_{2j+1} + t_{2j+2})$ into the
$2j$-th position for $1 \leq j < i$. By condition (ii) we can then choose 
$\lambda_2, \lambda_4, \ldots , \lambda_{2(i-1)}$, and $r$ sufficiently large so that $C_i > 0$.  
\end{proof}

In the following proposition, we will prove the required local estimate for $\Dir^k(u)$. We will 
use the definition of the $C_i$ outlined in the hypothesis of the above lemma.

\begin{proposition}\label{Dir^k_loc}
For $u \in \Sp_{k+1}$, given any $t_1, t_3, \ldots , t_{2k-1}, t_{2k} > 0$ 
we can choose positive $t_2, t_4, \ldots , t_{2(k-1)}, r$ sufficiently large so that 
$C_i > 0$ for $1\leq i \leq k$, and so that 
\begin{equation*}
\begin{aligned}
C_k\vert\vert &b_k^2 \Dir^k(u)\vert\vert^2_{B_{2e(r)r}} \leq 
\frac{1}{2t_k}\vert\vert b_r^{k+1}\Dir^{k+1}(u)\vert\vert^2_{B_{2e(r)r}} \\
&\qquad + \Big(\frac{t_1}{2} + \frac{t_2}{2}\Big)\Big(\frac{t_3}{2} + \frac{t_4}{2}\Big)\cdots\Big(\frac{t_{2k-1}}{2} + 
\frac{t_{2k}}{2}\Big)
C_1^{-1}C_2^{-1}\cdots C_{k-2}^{-1}C_{k-1}^{-1}
\vert\vert u\vert\vert^2_{B_{2e(r)r}}.
\end{aligned}
\end{equation*}
\end{proposition}
\begin{proof}
We will prove this by induction, the $k = 2$ case being done in Proposition \ref{Dir^2_loc_prop}. So assume
the Proposition is true for $k-1$. This means that for $u \in \Sp_k$, given any 
$t_1, t_3, \ldots , t_{2(k-1)-1}, t_{2(k-1)} > 0$ we can choose positive  $t_2, t_4, \ldots , t_{2(k-2)}, r$ 
sufficiently large so that $C_i > 0$ for $1\leq i \leq k-1$, and so that 
\begin{equation*}
\begin{aligned}
C_{k-1}\vert\vert b_r^{k-1}\Dir^{k-1}(u)\vert\vert^2_{B_{2r}} &\leq
\frac{1}{2t_{2(k-1)}}\vert\vert b_r^k\Dir^k(u)\vert\vert^2_{B_{2r}} \\
&+
\Big(\frac{t_1}{2} + \frac{t_2}{2}\Big)\Big(\frac{t_3}{2} + \frac{t_4}{2}\Big)\cdots\\ &\qquad \Big(\frac{t_{2(k-1)-1}}{2} + 
\frac{t_{2(k-1)}}{2}\Big)
C_1^{-1}C_2^{-1}\cdots C_{k-3}^{-1}C_{k-2}^{-1}
\vert\vert u\vert\vert^2_{B_{2r}}.
\end{aligned}
\end{equation*}  
It should be noted, as we observed in Remark \ref{rem_Dir_est}, 
that such an estimate holds merely for $u \in \Ck{\infty}(\cV)$.  
When we substitute this estimate into the derived estimate
for $\Dir^k(u)$ we will use this fact without explicit mention.

By Lemma \ref{constants}, we can choose $t_2, t_4, \ldots , t_{2(k-1)}$ and $r$  sufficiently large so that 
$C_i > 0$ for $1\leq i \leq k$. Furthermore, in  the proof  of that lemma, we showed how 
we could choose $t_2, t_4, \ldots , t_{2(k-1)}$ so that the $C_{k-1}$ would not depend on $t_{2(k-1)}$.
In particular, this means that for these $C_i > 0$ for $1 \leq i \leq k-1$ the above estimate
for $\Dir^{k-1}(u)$ holds with all the $C_i > 0$. Therefore, all that we need to do is show that 
the above estimate holds.

To obtain an estimate for $\Dir^k(u)$ we start by estimating $\vert\vert b_r^k\Dir^k(u)\vert\vert^2_{B_{2e(r)r}}$. 
We begin by fixing a smooth compactly supported function $\eta$ such that $\eta = 1$ on $B_{3e(r)r}$ and 
$\eta$  vanishing  outside of $B_{4e(r)r}$. We then have that
\begin{align*}
\langle b_r^k\Dir^k(\eta u), b_r^k\Dir^k(\eta u)\rangle &= \langle b_r^{2k}\Dir^k(\eta u), \Dir^k(\eta u)\rangle \\
&= \langle \Dir(b_r^{2k}\Dir^k(\eta u)), \Dir^{k-1}(\eta u)\rangle \\
&= \langle \Mul_{b_r^{2k}}(\Dir^k(\eta u)), \Dir^{k-1}(\eta u)\rangle
+ \langle b_r^{2k}\Dir^{k+1}(\eta u), \Dir^{k-1}(\eta u)\rangle \\
&= \langle \Mul_{b_r^{2k}}(\Dir^k(\eta u)), \Dir^{k-1}(\eta u)\rangle
+ \langle b_r^{k+1}\Dir^{k+1}(\eta u), b_r^{k-1}\Dir^{k-1}(\eta u)\rangle
\end{align*}
where the second equality follows from Proposition \ref{lip_sym}.

As $b_r^k$ is supported inside $\overline{B}_{2e(r)r}$, and $\eta = 1$ on $B_{3e(r)r}$ we have
\begin{multline*}
\langle b_r^k\Dir^k(u), b_r^k\Dir^k(u)\rangle_{B_{2e(r)r + \epsilon}} \\ =
\langle \Mul_{b_r^{2k}}(\Dir^k(u)), \Dir^{k-1}(u)\rangle_{B_{2e(r)r + \epsilon}}
+ \langle b_r^{k+1}\Dir^{k+1}(u), b_r^{k-1}\Dir^{k-1}(u)\rangle_{B_{2e(r)r + \epsilon}}
\end{multline*}
for any $\epsilon > 0$ sufficiently small. In particular, this then gives
\begin{multline*}
\langle b_r^k\Dir^k(u), b_r^k\Dir^k(u)\rangle_{B_{2e(r)r}} \\ =
\langle \Mul_{b_r^{2k}}(\Dir^k(u)), \Dir^{k-1}(u)\rangle_{B_{2e(r)r}}
+ \langle b_r^{k+1}\Dir^{k+1}(u), b_r^{k-1}\Dir^{k-1}(u)\rangle_{B_{2e(r)r}}.
\end{multline*}

Using Cauchy-Schwarz, the bound \eqref{L^2_bound_b_r^k}, and the fact that $b_r^{k-1}$ does not have any zeros  in  
$B_{2e(r)r}$, 
we have
\begin{align*}
\big\vert \langle \Mul_{b_r^{2k}}(\Dir^k(u)), &\Dir^{k-1}(u)\rangle_{B_{2e(r)r}}\big\vert \\  &= 
\Big\vert \Big\langle \frac{1}{b_r^{k-1}}\Mul_{b_r^{2k}}(\Dir^k(u)), b_r^{k-1}\Dir^{k-1}(u)\Big\rangle_{B_{2e(r)r}}\Big\vert \\
&\leq \frac{2k^2M^2}{t_{2k-1}r^2}\Big\vert\Big\vert \frac{b_r^{2k-1}}{b_r^{k-1}}\Dir^k(u)\Big\vert\Big\vert^2_{B_{2e(r)r}} + 
\frac{t_{2k-1}}{2}\vert\vert b_r^{k-1}\Dir^{k-1}(u)\vert\vert^2_{B_{2e(r)r}} \\
&= \frac{2k^2M^2}{t_{2k-1}r^2}\vert\vert b_r^k\Dir^k(u)\vert\vert^2_{B_{2e(r)r}} + 
\frac{t_{2k-1}}{2}\vert\vert b_r^{k-1}\Dir^{k-1}(u)\vert\vert^2_{B_{2e(r)r}}.
\end{align*}
Similarly we have
\begin{equation*}
\vert\langle b_r^{2k}\Dir^{k+1}(u), \Dir^{k-1}(u)\rangle_{B_{2e(r)r}}\vert \leq 
\frac{1}{2t_{2k}}\vert\vert b_r^{k+1}\Dir^{k+1}(u)\vert\vert^2_{B_{2e(r)r}} + 
\frac{t_{2k}}{2}\vert\vert b_r^{k-1}\Dir^{k-1}(u)\vert\vert^2_{B_{2e(r)r}}.
\end{equation*}
These two bounds imply
\begin{multline*}
\vert\vert b_r^k\Dir^k(u)\vert\vert^2_{B_{2e(r)r}} \leq 
\frac{2k^2M^2}{t_{2k-1}r^2}\vert\vert b_r^k\Dir^k(u)\vert\vert^2_{B_{2e(r)r}} + 
\Big(\frac{t_{2k-1}}{2} + \frac{t_{2k}}{2}\Big)\vert\vert b_r^{k-1}\Dir^{k-1}(u)\vert\vert^2_{B_{2e(r)r}} \\ +
\frac{1}{2t_{2k}}\vert\vert b_r^{k+1}\Dir^{k+1}(u)\vert\vert^2_{B_{2e(r)r}}.
\end{multline*}
We now substitute the estimate for $\vert\vert b_r^{k-1}\Dir^{k-1}(u)\vert\vert^2_{B_{2e(r)r}}$, as given in
the induction hypothesis, to obtain
\begin{align*}
 \vert\vert b_r^k&\Dir^k(u)\vert\vert^2_{B_{2e(r)r}} \leq 
\frac{2k^2M^2}{t_{2k-1}r^2}\vert\vert b_r^k\Dir^k(u)\vert\vert^2_{B_{2e(r)r}} \\&+
\Big(\frac{t_{2k-1}}{2} + \frac{t_{2k}}{2}\Big)C_{k-1}^{-1}\frac{1}{2t_{2(k-1)}}\vert\vert 
b_r^k\Dir^k(u)\vert\vert^2_{B_{2e(r)r}} \\
&+ 
\Big(\frac{t_1}{2} + \frac{t_2}{2}\Big)\Big(\frac{t_3}{2} + \frac{t_4}{2}\Big)\cdots\Big(\frac{t_{2(k-1)-1}}{2} + 
\frac{t_{2(k-1)}}{2}\Big)\Big(\frac{t_{2k-1}}{2} + \frac{t_{2k}}{2}\Big)
C_1^{-1}C_2^{-1}\\&\qquad\qquad\cdots C_{k-3}^{-1}C_{k-2}^{-1}C_{k-1}^{-1}\vert\vert u\vert\vert^2_{B_{2r}} \\
&+ \frac{1}{2t_{2k}}\vert\vert b_r^{k+1}\Dir^{k+1}(u)\vert\vert^2_{B_{2e(r)r}}. 
\end{align*}
 This can be re-written as 
\begin{align*}
 \Bigg(1 - & \frac{2k^2M^2}{t_{2k-1}r^2} - 
\Big(\frac{t_{2k-1}}{2} + \frac{t_{2k}}{2}\Big)C_{k-1}^{-1}\frac{1}{2t_{2(k-1)}}\Bigg) \vert\vert 
b_r^k\Dir^k(u)\vert\vert^2_{B_{2e(r)r}} \\&\leq 
\frac{1}{2t_{2k}}\vert\vert b_r^{k+1}\Dir^{k+1}(u)\vert\vert^2_{B_{2e(r)r}} +
\Big(\frac{t_1}{2} + \frac{t_2}{2}\Big)\Big(\frac{t_3}{2} + \frac{t_4}{2}\Big) \\ &\qquad\quad\cdots\Big(\frac{t_{2(k-1)-1}}{2} + 
\frac{t_{2(k-1)}}{2}\Big)\Big(\frac{t_{2k-1}}{2} + \frac{t_{2k}}{2}\Big)
C_1^{-1}C_2^{-1}\cdots C_{k-3}^{-1}C_{k-2}^{-1}C_{k-1}^{-1}\vert\vert u\vert\vert^2_{B_{2e(r)r}}
\end{align*}
By definition, the coefficient on the left hand side is precisely $C_k$, and so we have established the
estimate for $\Dir^k(u)$.
\end{proof}

The local estimates we obtain, from the above proposition, can be promoted to a global estimate by taking $r 
\rightarrow \infty$. All we have to note is that we have explicit formulas for the constants. We know that 
$C_1(r, t_1) = \Big(1 - \frac{2M^2}{t_1r^2}\Big)$, and in general
\begin{equation*}
C_i(r, t_1, t_2, \ldots , t_{2i-1}, t_{2i}) = \Bigg(1 - \frac{2i^2M^2}{t_{2i-1}r^2} - 
\Big(\frac{t_{2i-1}}{2} + \frac{t_{2i}}{2}\Big)C_{i-1}^{-1}\frac{1}{2t_{2(i-1)}}\Bigg). 
\end{equation*}
Therefore, it follows that as
$r \rightarrow \infty$ we have that $C_1 \rightarrow 1$, and in general
\begin{equation*}
C_i \rightarrow C_i^0 := \Bigg(1 - \Big(\frac{t_{2i-1}}{2} + \frac{t_{2i}}{2}\Big)(C_{i-1}^0)^{-1}\frac{1}
{2t_{2(i-1)}}\Bigg).
\end{equation*}

One can then prove an  analogous result  to  Lemma \ref{constants} for the $C_i^0$. We will not do this
as the proof is exactly the same as the proof of Lemma \ref{constants}. Using this fact, and taking
$r \rightarrow \infty$ in  the  local  estimate in Proposition \ref{Dir^k_loc}, we obtain the following theorem.

\begin{theorem}\label{Dir^k_global}
Fix $u \in \Sp_{k+1}$, given positive $t_1, t_3, t_5, \ldots , t_{2k-3}, t_{2k-1}, t_{2k}$ we can find
positive $t_2, t_4, \ldots , t_{2(k-2)}, t_{2(k-1)}$ so that  $C_i^0 > 0$ for all $1 \leq i \leq k$, and
so that
\begin{align*}
&C_{k}^0\vert\vert \Dir^{k}(u)\vert\vert^2 \leq 
\frac{1}{2t_{2k}}\vert\vert \Dir^{k+1}(u)\vert\vert^2  +
\Big(\frac{t_1}{2} + \frac{t_2}{2}\Big)\\&\qquad\cdots\Big(\frac{t_{2(k-1)-1}}{2} + 
\frac{t_{2(k-1)}}{2}\Big)\Big(\frac{t_{2k-1}}{2} + \frac{t_{2k}}{2}\Big)
(C_1^0)^{-1}(C_2^0)^{-1}\cdots (C_{k-2}^0)^{-1}(C_{k-1}^0)^{-1}\vert\vert u\vert\vert^2.
\end{align*}
In particular $\vert\vert \Dir^k(u)\vert\vert < \infty$, and $\Sp_{k+1} \subseteq \Sp_k$.

\end{theorem}

With the aid of this theorem,  we conclude the paper
by presenting the following proof of Theorem \ref{Thm:SecondMain}. 

\begin{proof}[Proof of Theorem \ref{Thm:SecondMain}]
The above theorem shows us that $\Sp_k \subseteq \Sp_{k-1} \subseteq \ldots \subseteq \Sp_2 \subseteq \Sp_1$  for $1 \leq k \leq m+1$. 
 Fixing  $u, v \in \Sp_k$, we have using Theorem \ref{D_essen_self} that 
\begin{equation*}
\langle \Dir^ku, v\rangle = \langle \Dir^{k-1}u, \Dir v\rangle = \langle \Dir^{k-2}u, \Dir^2v\rangle = \ldots =
\langle \Dir u, \Dir^{k-1}v\rangle = \langle u, \Dir^kv\rangle. 
\end{equation*}
This implies that $\Dir^k$ satisfies the negligible boundary condition \eqref{lneg}
for $l = k$ and the conclusion follows from Theorem \ref{Thm:FirstMain}.
\end{proof}

\bibliographystyle{amsplain}
\providecommand{\bysame}{\leavevmode\hbox to3em{\hrulefill}\thinspace}
\providecommand{\MR}{\relax\ifhmode\unskip\space\fi MR }
\providecommand{\MRhref}[2]{%
  \href{http://www.ams.org/mathscinet-getitem?mr=#1}{#2}
}
\providecommand{\href}[2]{#2}

\setlength{\parskip}{0mm}

\end{document}